\documentclass[oneside]{amsart}
\usepackage[utf8]{inputenc}
\usepackage[a4paper]{geometry}
\usepackage{amsmath, amssymb, amsthm}
\usepackage{latexsym} 
\usepackage{dsfont}
\usepackage[flushleft]{paralist}
\usepackage[all]{xy}
\usepackage{tikz}
\usetikzlibrary{cd,decorations.pathreplacing}
\usepackage{stmaryrd} 
\usepackage{colonequals}
\usepackage{enumitem}
\usepackage{soul}
\usepackage[OT2,T1]{fontenc}
\usepackage[style=alphabetic,backend=biber]{biblatex}
\usepackage{hyperref}
\usepackage{xurl}
\hypersetup{breaklinks=true}

\tikzcdset{scale cd/.style={every label/.append style={scale=#1},
    cells={nodes={scale=#1}}}}
\DeclareSymbolFont{cyrletters}{OT2}{wncyr}{m}{n}
\DeclareMathSymbol{\Zhe}{\mathalpha}{cyrletters}{"11} 
\DeclareMathOperator{\Tr}{Tr}
\DeclareMathOperator{\Hom}{Hom}

\DeclareMathOperator{\res}{res}
\DeclareMathOperator{\cd}{cd}
\DeclareMathOperator{\coker}{coker}
\newcommand{\Z}{\mathds{Z}}

\newcommand{\Q}{\mathds{Q}}

\newcommand{\p}{\mathfrak{p}}

\newcommand{\cK}{\mathcal{K}}
\newcommand{\cL}{\mathcal{L}}
\newcommand{\cF}{\mathcal{F}}
\newcommand{\Gal}{\textup{Gal}}

\newcommand{\Image}{\textup{Im}}

\newcommand{\Sel}{\textup{Sel}}

\renewcommand{\ss}{\mathrm{ss}}
\newcommand{\ord}{\mathrm{ord}}
\newcommand{\G}{\mathcal G}
\newcommand{\Selz}{\Sel^{\Sigma_0}}
\newcommand{\Hiw}{H^1_{\mathrm{Iw}}}
\newcommand{\Hpm}{H^\pm_{\mathrm{Iw}}}

\usepackage[capitalise, noabbrev]{cleveref}
\newtheorem{lemma}{Lemma}[section] 
\newtheorem{prop}[lemma]{Proposition}

\newtheorem{thm}[lemma]{Theorem} 
\newtheorem*{thm*}{Theorem} 
\newtheorem{cor}[lemma]{Corollary}

\Crefname{prop}{Proposition}{Propositions}

\theoremstyle{definition}
\newtheorem{def1}[lemma]{Definition} 

\newtheorem{rem}[lemma]{Remark}

\begin{document}
\title[Plus/minus Selmer groups in weakly ramified extensions]{On the cohomology of plus/minus Selmer groups of supersingular elliptic curves in weakly ramified base fields}\author[B. Forrás]{Ben Forrás} 
\address[Forrás]{Institut für Theoretische Informatik, Mathematik und Operations Research, Universität der Bundeswehr München, Werner-Heisenberg-Weg 39, 85577 Neubiberg, Germany}
\email{ben.forras@unibw.de} 
\author[K.~Müller]{Katharina Müller}
\address[Müller]{Institut für Theoretische Informatik, Mathematik und Operations Research, Universität der Bundeswehr München, Werner-Heisenberg-Weg 39, 85577 Neubiberg, Germany}
\email{katharina.mueller@unibw.de}
\begin{abstract}
    Let $E/\mathbb{Q}$ be an elliptic curve and let $p\ge 5$ be a prime of good supersingular reduction. We generalize results due to Meng Fai Lim proving Kida's formula and integrality results for characteristic elements of signed Selmer groups along the cyclotomic $\Z_p$-extension of weakly ramified base fields $K/\Q_p$. 
\end{abstract}

\maketitle

\section{Introduction}
The Iwasawa theory of elliptic curves at good supersingular primes with $a_p=0$ was initiated by Kobayashi in his seminal work \cite{Kobayashi}, {where $a_p=1+p-|\widetilde E(\mathbb F_p)|$, and $|\widetilde E(\mathbb F_p)|$ is the the number of points in the reduction of the curve modulo $p$}. Working with the cyclotomic $\Z_p$-extension of $\Q$, he defined plus/minus norm subgroups as well as a corresponding signed Selmer group, and used Coleman theory to prove that it is cotorsion over the Iwasawa algebra. This involved writing down an explicit power series, and using Honda theory to associate with it a formal group law which turns out to be isomorphic to the formal group of the elliptic curve. Crucially, this formal group law has no $p$-torsion over the cyclotomic tower. Moreover, one can explicitly construct a norm coherent sequence of points along the tower generating the corresponding even/odd norm subgroups.

This strategy was generalised, first by Iovita--Pollack \cite{IP} to extensions of $\Q$ where $p$ splits completely, then by B.D.~Kim to extensions of $\Q$ with $p$ unramified \cite{BDK2007,BDK2013}. Torsion properties and finite submodules of the Pontryagin dual of the signed Selmer groups were further studied by Kitajima--Otsuki \cite{KO} and Lei--Lim \cite{LeiLim}, among others.

In {a} recent work, Lim \cite{mengfai} established cohomological triviality of the plus/minus subgroups, studied projectivity of signed Selmer groups, and used these results to deduce a Kida formula as well as to prove an integrality result for characteristic elements for the Pontryagin dual of signed Selmer groups. The aim of this paper is to generalise these results by allowing tame ramification at supersingular primes under certain assumptions. 

Before we state our main results, let us fix some notation. Let $p$ be an odd rational prime, $F/F'/\Q$ number fields with $p$ splitting completely in $F'$, $K/F$ a finite Galois extension. Let $K_\infty/K$ by the cyclotomic $\Z_p$-extension. Let $G\colonequals\Gal(K_\infty/F_\infty)$ and $\Gamma\colonequals\Gal(F_\infty/F)$. Then we have an isomorphism $\G\colonequals\Gal(K_\infty/F)=G\rtimes \Gamma$. {We fix once and for all a lift $\Gamma'$ of $\Gamma$ in $\mathcal{G}$ such that the resitiction induces an isomorphism $\Gamma'\cong \Gamma$. By abuse of notation we denote $\Gamma'$ by $\Gamma$.} Let $\Lambda\colonequals\Z_p\llbracket\Gamma\rrbracket$ and $\Lambda(\G)\colonequals\Z_p\llbracket\G\rrbracket$ be the relevant Iwasawa algebras. Recall that a finite Galois extension of local fields is called weakly ramified if its second ramification group vanishes -- so wild ramification is allowed, but only in the first ramification group. Let $E/F'$ be an elliptic curve satisfying the following assumptions:
\begin{itemize}
    \item[(S1)] $E$ has good reduction at all $p$-adic places of $F'$;
    \item[(S2)] there is a $p$-adic place with supersingular reduction;
    \item[(S3)] each $p$-adic supersingular place $u$
    \begin{enumerate}[label=\roman*)]
        \item has ramification index $e_u(K/F')$ that is not divisible by $p^2-1$ in $K/F'$,
        \item  fulfills the following condition:  $K_u$ is contained in the compositum of an at most weakly ramified extension $\cK'/\Q_p$ {and the cyclotomic $\Z_p$-extension $\Q_{p,\infty}$, such that $\cK'$ does not intersect the cyclotomic extension $\Q_{p,\infty}$},
        \item satisfies $a_u=0$, where $a_u=1+p-|\widetilde E_u(\mathbb F_p)|$ and $\widetilde E_u$ is the reduction of $E$ at $u$.
    \end{enumerate}
\end{itemize}
Furthermore, we fix a finite set of places $\Sigma$ of $F$ subject to certain standard conditions, as well as a subset $\Sigma'\subset \Sigma$ consisting only of some non-$p$-adic places. Let $\vec s$ be a tuple consisting of a sign in $\{+,-\}$ for each $p$-adic supersingular place in $\Sigma$. Analogously to the works cited above, we may define signed Selmer groups $\Sel^{\vec s}(E/K_\infty)$ as well as non-primitive versions $\Sel^{\vec s}_{\Sigma'}(E/K_\infty)$ thereof; for the precise definitions, we refer to \cref{sec:definition-of-signed-Sel,sec:non-primitive-Sel}. 

Our first main result is the following:

\begin{thm*}[Kida formula, \cref{thm:Kida}]
    Assume that $\Sel^{\vec s}(E/K_\infty)$ is $\Lambda$-cotorsion and that $\theta(X^{\vec s}(E/K_\infty))\le 1$. 
    Let $P_1\subset \Sigma'$ be the primes where $E$ has split multiplicative reduction and let $P_2$ be the set of primes in $\Sigma'$ where $E$ has good reduction and $E(K)[p]\neq 0$. Then {the Iwasawa $\lambda$- resp. $\mu$-invariants of the $\Lambda$-modules $X^{\vec s}(E/K_\infty)$ and $X^{\vec s}(E/L_\infty)$ are related as follows:}
    \begin{align*}
        \lambda\left(X^{\vec s}(E/K_\infty)\right) &= [K_\infty:L_\infty]\lambda\left(X^{\vec s}(E/L_\infty)\right)+\sum_{v\in P_1}(e_v-1)+2\sum_{w\in P_2}(e_v-1), \\
        \mu\left(X^{\vec s}(E/K_\infty)\right)&=[K_\infty:L_\infty]\mu\left(X^{\vec s}(E/L_\infty)\right).
    \end{align*}
\end{thm*}

Here $\theta(X^{\vec s}(E/K_\infty))$ denotes the largest exponent in the standard decomposition of the maximal $p$-power-torsion submodule of $X^{\vec s}(E/K_\infty)$, that is, if $X^{\vec s}(E/K_\infty)[p^\infty]\to\bigoplus_{i\in I} \Lambda/p^{m_i}\Lambda$ is a pseudo-isomorphism, then $\theta(X^{\vec s}(E/K_\infty)[p^\infty])\colonequals\max\{m_i:i\in I\}$.
The proof of this Kida formula is by calculating certain Herbrand quotients: this is a method due to Iwasawa \cite[\S9]{IwasawaRH}. The method was used by Hachimori--Matsuno \cite{HachimoriMatsuno} to prove a Kida formula for elliptic curves with good ordinary reduction and $\mu=0$. The weakening of the $\mu=0$ assumption to $\theta\le 1$ above is along the lines of Hachimori--Sharifi's Kida formula for CM fields \cite{hachimori-sharifi}. Our result is a direct generalisation of Lim's Kida formula \cite[Proposition~5.2]{mengfai}.

We come to our second main result. Let $\Gamma_0$ be an open subgroup of $\Gamma$ that is central in $\G$. Recall that an ordinary $p$-adic place $v$ of $F$ is called non-anomalous if for all places $w$ of $K$ above $v$, we have $p\nmid|\widetilde E(k_w)|$, where $k_w$ is the residue field of the local field $K_w$, and $\widetilde E$ is the reduction of $E$.

\begin{thm*}[Integrality of characteristic elements, \cref{thm:integrality}]
    Suppose that $\Sigma'$ contains all non-$p$-adic places in $\Sigma$ whose inertia degree in $K/F$ is divisible by $p$. Further suppose that $X^{\vec s}(E/K_\infty)$ is $\Lambda$-torsion, and that every ordinary $p$-adic place is either non-anomalous or ramifies tamely in $K/F$.
    Then there exists a characteristic element $\xi_{E,\Sigma'}$ of $X^{\vec s}_{\Sigma'}(E/K_\infty)$. For every {graduated} $\Lambda(\Gamma_0)$-order $\mathfrak M$ of $\mathcal Q(\G)$ containing $\Lambda(\G)$, the characteristic element $\xi_{E,\Sigma'}$ is in the image of the natural map $\mathfrak M\cap \mathcal Q(\G)^\times\to K_1(\mathcal Q(\G))$.
\end{thm*}
{Graduated orders constitute a generalisation of maximal orders; the precise definition will be recalled in \cref{sec:integrality}.}
The characteristic elements above are essential for the main conjecture of Coates et al.~\cite{GL2}. An integrality result {for maximal orders} was established by Nichifor--Palvannan \cite[\S5]{NichiforPalvannan} for elliptic curves with good ordinary or split multiplicative reduction admitting a cyclic isogeny of order $p^2$. They also developed an algebraic tool for studying characteristic elements of $\Lambda(\G)$-modules admitting a projective resolution of length~1 in the case when $\G$ is a direct product of $\Gamma$ and a finite group; this method was generalised to semidirect products by the first named author in \cite{W,EpAC}.

In the process of establishing the results above, we make extensive use of the properties of $p$-torsion freeness and cohomological triviality mentioned above, which are easy consequences of our assumptions (S1--S3). 
An important distinction between the present article and preexisting work is that we don't have access to a norm compatible system of points, as these are only known to exist in the unramified setting. While the outline of our proofs rely heavily on the work of Lim in the unramified case, our computation of the relevant cohomology groups is a more laborious endeavour. Indeed, if $\cK/\Q_p$ is a finite unramified extension with cyclotomic $\Z_p$-extension $\cK_\infty$, then the existence of a norm coherent sequence gives rise to a short exact sequence $\widehat E(\cK)\hookrightarrow \widehat E^+(\cK_\infty)\oplus\widehat E^-(\cK_\infty)\twoheadrightarrow \widehat E(\cK_\infty)$, whereas we make no claim of such a sequence existing, and compute the relevant cohomology groups by studying the module-theoretic properties of certain plus/minus Iwasawa cohomology groups. 
Note that there is no apparent way of directly generalising the construction of norm coherent points to the ramified setting. Indeed, using Kobayashi's approach would involve Honda theory, which requires working over an unramified extension. On the other hand, lifting a norm coherent sequence from the cyclotomic tower over an unramified extension to a ramified extension thereof would not preserve the desired norm relations.

We remark that recently, Kataoka has constructed a framework for deriving Kida formul{\ae} via Selmer complexes \cite{kataoka2024kida}. This differs from the present work in its setting: indeed, \S4.5 of op.cit. treats supersingular elliptic curves for abelian extensions only, whereas \S5.2 of op.cit. treats some non-abelian cases, but only for $\lambda$-invariants. It would be interesting to see if the results of the present article can be interpreted in Kataoka's framework.

It is a natural question to ask whether our results could be generalised to modular forms. Indeed, several elements of the theory recalled above, such as Coleman maps, have been constructed for modular forms by Lei--Loeffler--Zerbes \cite{LLZ} using $p$-adic Hodge theory and Wach modules, with the relationship to Kobayashi's signed Selmer groups explained in \cite{LeiZerbes}. 

The paper is structured as follows. \Cref{sec:prelim} consists of a collection of general facts related to Galois cohomology. We treat local cohomology groups in \cref{sec:local-cons}. In \cref{sec:global-cons}, we define signed Selmer groups in the above setting, and study torsion properties as well as finite submodules of their Pontryagin duals. We establish projectivity results in \cref{sec:projectivity}. The Kida formula and the integrality result stated above are proven in \cref{sec:Kida} and \cref{sec:integrality}, respectively. Finally in \cref{sec:cong}, we study the relationship between the Iwasawa invariants of elliptic curves whose $p$-torsion points are isomorphic as Galois modules.

\section*{Acknowledgments}
The authors would like to thank Andreas Nickel for various and extensive comments on a draft version of this paper, as well as Cornelius Greither for helpful remarks. They also thank the referees for insightful comments and useful suggestions.

\section*{Data availability statement}
All relevant data are contained within the manuscript.

\section{Preliminaries on Galois cohomology} \label{sec:prelim}
Let $\mathcal{G}\colonequals G\rtimes \Gamma$, where $\Gamma\cong \Z_p$ and $G$ is a finite group.  {We fix once and for all a lift $\Gamma'$ of $\Gamma$ in $\mathcal{G}$ such that $\Gamma'\cong \Gamma$ by restriction. Let $M$ be a $\Z_p\llbracket \mathcal{G}\rrbracket$-module. Note that the action of $\mathcal{G}$ on $M_G$ factors through $\Gamma$ and coincides with the action of $\Gamma'$. By abuse of notation we will frequently write $\Gamma$ instead of $\Gamma'$.} Let $\Lambda\colonequals\Z_p\llbracket \Gamma\rrbracket$ and $\Lambda(\mathcal{G})\colonequals\Z_p\llbracket\mathcal{G}\rrbracket$. If $G$ is a cyclic group and $M$ is a $G$-module with finite cohomology groups, we write $h_G(M)\colonequals |H^2(G,M)|/|H^1(G,M)|$ for the Herbrand quotient.
\begin{lemma}
\label{hachimori-sharifi-1}
    Assume that $M$ is an $\mathbb{F}_p[G]$-module and that $G$ is a cyclic $p$-group. Assume furthermore that $H^i(G,M)$ is finite for all $i\ge 1$. Then the Herbrand quotient $h_G(M)$ is trivial.
\end{lemma}
\begin{proof}
    This is a straightforward generalization of \cite[lemma 2.2]{hachimori-sharifi}.
\end{proof}
For a finitely generated $\Lambda$-module $M$, let $\theta(M)$ denote the largest exponent in the standard decomposition of the maximal $p$-power-torsion submodule $M[p^\infty]$ of $M$ up to pseudo-isomorphism. In formul\ae: there is a pseudo-isomorphism $M[p^\infty]\to\bigoplus_{i\in I} \Lambda/p^{m_i}\Lambda$, and $\theta(M)=\max\{m_i:i\in I\}$.
\begin{lemma}
\label{hachimorio-sharifi-2}
Assume that $G$ is a cyclic $p$-group.
    Let $M$ be a finitely generated $\Lambda(\G)$-module that is torsion as a $\Lambda$-module. Assume that $H^i(G,M)$ is finitely generated over $\Z_p$ for all $i\ge 1$ and that $\theta(M)\le 1$. Then $h_G(M{[p^\infty]})=1$ and $\mu(M)=\vert G\vert\mu(M_G)$.
\end{lemma}
\begin{proof}
    This is basically \cite[Lemma 2.4]{hachimori-sharifi}, but there it's assumed that the actions of $\Gamma$ and $G$ commute. Consider the short exact sequence
    \[0\to M{[p^\infty]}\to M\to Z\to 0,\]
    where $Z$ is $\Z_p$-free {and finitely generated over $\Z_p$}. Taking cohomology with respect to $G$, and using that $H^i(G,Z)$ is finitely generated over $\Z_p$ for all $i$, we obtain that $H^i(G, M{[p^\infty]})$ is finitely generated over $\Z_p$  for all $i\ge 1$. 
    As $\theta(M)\le 1$, we see that the natural map $M{[p^\infty]}\to M{[p^\infty]}/pM{[p^\infty]}$ is actually a surjective pseudo-isomorphism. Therefore $h_G(M{[p^\infty]})=h_G(M{[p^\infty]}/pM{[p^\infty]})$, and $H^i(G,M{[p^\infty]}/pM{[p^\infty]}))$ is finitely generated over $\Z_p$ for all $i\ge 1$. This implies that $H^i(G,M{[p^\infty]}/pM{[p^\infty]})$ is finite for all $i\ge 1$, and \cref{hachimori-sharifi-1} shows that $h_G(M{[p^\infty]}/pM{[p^\infty]})=1$.

    It remains to show the claim on $\mu$-invariants. As $H^i(G,M{[p^\infty]}/pM{[p^\infty]})$ is finite for all $i\ge 1$, we see that the three modules $(M{[p^\infty]}/pM{[p^\infty]})_G$, $(M{[p^\infty]}/pM{[p^\infty]})^G$ and $N_G(M{[p^\infty]}/pM{[p^\infty]})$ are pseudo-isomorphic as $\Lambda$-modules, where $N_G\colonequals \sum_{g\in G} g {\in\Z_p[G]}$. In particular, they have the same $\mu$-invariant. Let $\tau$ be a generator of $G$, and let $I_G\colonequals\Lambda(\mathcal{G})(\tau-1)$. As $G$ is a normal subgroup of $\G$, the ideal $I_G^k$ is generated by $(\tau-1)^k$. {The $\mu$-invariant can then be expressed as the sum of the $\mu$-invariants of the quotients in the filtration induced by $I_G$. Each $\mu$-invariant in this sum can be estimated from below by the $\mu$-invariant in the last module in the filtration:} 
    \begin{align*}
        \mu(M)&=\mu(M{[p^\infty]}/pM{[p^\infty]})=\sum_{k=0}^{\vert G\vert-1}\mu\big(I_G^k(M{[p^\infty]}/pM{[p^\infty]}\big)\big/I_G^{k+1}\big(M{[p^\infty]}/pM{[p^\infty]}\big)\\
        &\ge \vert G\vert \cdot \mu\!\left(I_G^{\vert G\vert-1}(M{[p^\infty]}/pM{[p^\infty]})\right)=\vert G\vert \cdot \mu\big(N_G(M{[p^\infty]}/pM{[p^\infty]})\big)\\
        &=\vert G\vert \cdot \mu\big((M{[p^\infty]}/pM{[p^\infty]})_G\big)=\vert G\vert \cdot \mu(M_G).
    \end{align*}
    {The last equality follows from the fact that $\theta(M_G)\le 1$. Thus, $\mu(M_G)=\mu(M_G/pM_G)$. Furthermore, $M_G/pM_G\cong M/(\tau-1,p)M\cong (M/pM)_G$. We have an exact sequence
    \[Z^{G}\to (M[p^\infty]/pM[p^\infty])_G\to (M/pM)_G\to (Z/pZ)_G.\]
    The first and last term are finite. Therefore, $\mu((M/pM)_G)=\mu((M[p^\infty]/pM[p^\infty])_G)$.}

    On the other hand, $\big(M{[p^\infty]}/pM{[p^\infty]}\big)\big/\big(I_G(M{[p^\infty]}/pM{[p^\infty]})\big)$ surjects onto $N_G\big(M{[p^\infty]}/pM{[p^\infty]}\big)$ with finite kernel. {Once again writing $\mu(M)$ as a sum, each summand can therefore be estimated from above:}
    \begin{align*}
        \mu(M)&=\mu\big(M{[p^\infty]}/pM{[p^\infty]}\big)=\sum_{k=0}^{\vert G\vert-1} \mu\big(I_G^k(M{[p^\infty]}/pM{[p^\infty]}\big)\big/I_G^{k+1}\big(M{[p^\infty]}/pM{[p^\infty]})\big)\\
        &\le \vert G\vert \cdot \mu\big(N_G(M{[p^\infty]}/pM{[p^\infty]})\big)=\vert G\vert \cdot \mu\big((M{[p^\infty]}/pM{[p^\infty]})_G\big)=\vert G\vert \cdot \mu(M_G).
    \end{align*}
    {We obtain the desired equality.}
\end{proof}

\begin{lemma}
\label{perrin-riou}
    Let $E$ be an elliptic curve over {$\Q_p$} that is supersingular at $p$. Let $\cK/\Q_p$ be a finite extension and $\cK_\infty$ its cyclotomic $\Z_p$-extension, and assume that for all $n\ge0$, the group $E(\cK_n)$ has no $p$-torsion. Then 
    \[H^1({\cK_\infty},E[p^\infty])=E(\cK_\infty)\otimes \Q_p/\Z_p.\]
\end{lemma}
Later on, we will assume $a_p=0$; note that for $p\ge5$, this already follows from $p\mid a_p$ by the Hasse bound. In our applications of the \namecref{perrin-riou}, the $p$-torsion freeness assumption will be satisfied by \cref{lem:torsion-freeness}.
\begin{proof}
Let $k$ be the residue field of $\cK$. As $E$ is a supersingular elliptic curve, $p\textup{-rank}(E(k))=0$, where the $p$-rank is defined as in \cite{schneider}. It follows from \cite[Theorem 2]{schneider} that the group of universal norms $\bigcap_n N_{\cK_n/\cK}(E(\cK_n))$ is trivial. In particular, local duality \cite[\nopp I.3.4]{milne2006} implies that 
\[H^1(\cK_\infty,E)[p^\infty]=\varinjlim_n H^1({\cK_n},E)[p^\infty]=\textup{Hom}\left(\varprojlim_n E(\cK_n),\Q_p/\Z_p\right)=0,\]
which in turn implies that $H^1({\cK_\infty},E[p^\infty])=E(\cK_\infty)\otimes\Q_p/\Z_p$. 
\end{proof}

Recall that a finitely generated $\Z_p[G]$-module is called strictly quasi-projective if it admits a pseudo-isomorphism $X\to Y$ to a projective $\Z_p[G]$-module $Y$. A finitely generated $\Z_p[G]$-module is called quasi-projective if there exist finitely generated strictly quasi-projective $\Z_p[G]$-modules $X_1$, $X_2$ such that there is an exact sequence $0\to X_1\to X_2\to X\to 0$ of $\Z_p[G]$-modules.

Let $Q$ be a finite cyclic group of order coprime to $p$. Let $\Q_p(\mu_{|Q|})$ denote the field obtained by adjoining all roots of unity of order $|Q|$ to $\Q_p$, and let $\mathcal O_{\Q_p(\mu_{|Q|})}$ denote its ring of integers. For a (necessarily $1$-dimensional) character $\varepsilon$ of $Q$, we have an associated idempotent \[e(\varepsilon)=|Q|^{-1}\sum_{q\in Q} \varepsilon(q^{-1}) q\in \mathcal O_{\Q_p(\mu_{|Q|})}[Q].\] If $X$ is a finitely generated $\Z_p[G]$-module and $Q\le G$ is a finite cyclic subgroup of order coprime to $p$, then the $\varepsilon$-component of $X$ is defined as $X^\varepsilon\colonequals e(\varepsilon) (X\otimes_{\Z_p} \mathcal O_{\Q_p(\mu_{|Q|})})$. We similarly define the $\varepsilon$-component of the Pontryagin dual $S=X^\lor=\Hom(X,\Q_p/\Z_p)$ by $S^\varepsilon\colonequals e(\varepsilon)(S\otimes_{\Z_p}\mathcal O_{\Q_p(\mu_{|Q|})})$.

\begin{prop}
\label{criterion-quasi-projective}
    Let $S$ be a discrete $\Lambda(\mathcal{G})$-module, and let $X$ be its Pontryagin dual. Assume that $X$ is finitely generated and torsion over $\Lambda$, and that $\theta(X)\le 1$. Suppose that for all cyclic subgroups $C=PQ$ of $G$, where $P$ is a $p$-group and $Q$ has order coprime to $p$, the cohomology groups $H^i(P,X)$ are finite for all $i\ge 1$, and $h_P(S^\varepsilon)=1$ for all (one-dimensional) characters $\varepsilon$ of $Q$. Then ${X/X[p]}$ is quasi-projective as $\Z_p[G]$-module.
\end{prop}
\begin{proof}
This is \cite[Proposition 2.2.1]{Greenberg2011} in the case that $G$ commutes with $\Gamma$. The only place where this assumption is used is while proving $h_P(S^\varepsilon)=h_P(pS^\varepsilon)$. We now prove this without assuming that $G$ and $\Gamma$ commute. 

Let $M\colonequals S[p]^\varepsilon$. Then $M$ is a $\mathbb{F}_p[G]$-module and $M^\vee=X^\varepsilon/pX^\varepsilon$. As $\theta(X)\le 1$, the module $pX$ is finitely generated over $\Z_p$. We have a long exact sequence
\[\ldots\to H^i(P,X^\varepsilon)\to H^i(P,X^\varepsilon/pX^\varepsilon)\to H^{i+1}(P,pX^\varepsilon)\to \ldots\]
The first and the last term are finitely generated over $\Z_p$ for all $i\ge 1$. Thus the middle term is finite for all $i\ge 1$. We can now apply \cref{hachimori-sharifi-1} to conclude that $1=h_P(M^\lor)=h_P(M)^{-1}$, which in turn implies that $1=h_P(S^\varepsilon)=h_P(pS^\varepsilon)$ as desired.
\end{proof}

For a local field $\cK$, we write $\mathcal O_\cK$ resp. $m_\cK$ for its ring of integers resp. maximal ideal. If $\cL/\cK$ is a finite Galois extension of local fields, and $\mathfrak F$ is a formal group defined over $\mathcal O_\cK$, we let $\Tr^{\mathfrak F}: \mathfrak F(m_\cL)\to \mathfrak F(m_\cK)$ denote the norm with respect to $\mathfrak{F}$, as defined in \cite[\nopp 2.2.3]{Hazewinkel}. In the following special case, the $\mathfrak F$-norm is surjective:
\begin{prop}
\label{trace surjective}
    Let $\cL/\cK$ be a cyclic extension of local fields, and let $\mathfrak{F}$ be a formal group defined over $\mathcal{O}_\cK$ of height at least $2$. Let $t$ be the last ramification jump (i.e. $G_t\neq 1$ and $G_{t+1}=1$), and assume that $t\le 1$ (i.e. $\cL/\cK$ is at most weakly ramified). Then $\mathfrak{F}(m_\cK)=\textup{Tr}^{\mathfrak{F}}(\mathfrak{F}(m_\cL))$. 
\end{prop}
\begin{proof}
    It suffices to consider the case where $[\cL:\cK]=\ell$ is a prime. If $\cL/\cK$ is at most tamely ramified, this is handled in \cite[Proposition~3.10]{EllerbrockNickel}. Thus we will assume $t=1$ for the rest of the proof.
    
    By \cite[Proposition 3.5]{EllerbrockNickel}, for every $n\ge1$ there are well defined maps
    \[\alpha_n\colon \mathfrak{F}\left(m_\cL^{\psi(n)}\right)\Big/\mathfrak{F}\left(m_\cL^{\psi(n)+1}\right)\to \mathfrak{F}\left(m_\cK^n\right)\Big/\mathfrak{F}\left(m_\cK^{n+1}\right)\]
    where $\psi$ is as in \cite[IV§3]{serre-local-fields}.
    It follows from \cite[Corollary 3.6]{EllerbrockNickel} that these maps are isomorphisms for $n\ge t+1$. We want to show that it is also an isomorphism for $n=t=1$; . So assume $t=1$. By \cite[Corollary~2.4.2]{Hazewinkel}, there are coefficients $a_i\in\mathcal O_\cK$ such that
    \[\textup{Tr}^{\mathfrak{F}}(x)\equiv \textup{Tr}(x)+\sum_{i=1}^\infty a_i\textup{Nm}(x)^i\pmod{\textup{Tr}(x^2)}\]
    and the coefficient $a_1$  has positive valuation as $\mathfrak{F}$ has height at least $2$; Hazewinkel's result is valid in the case $\ell=p$, which is the case for $t=1$. 
    It follows that $\alpha_n$ is surjective for $n=1$ if and only if $\textup{Tr}_{\cL/\cK}(m_\cL)=m_\cK$. By \cite[V\S3, Lemma~4]{serre-local-fields}, this is always the case if 
    \[ 2>\frac{2(\ell-1)+1}{\ell}\ge 1,\]
    This chain of inequalities is trivially true and the map $\alpha_n$ is an isomorphism for all $n\ge1$. Now \cite[V\S1, Lemma~2]{serre-local-fields} implies the desired claim for $t=0$ and for $t=1$. 
\end{proof}
\begin{cor}
\label{cor-surjective}
    Keep the assumptions of \cref{trace surjective}. Then $H^i(\Gal(\cL/\cK),\mathfrak{F}(m_\cL))=0$ for all $i>0$. 
\end{cor}
\begin{proof}
 Let $H\colonequals\Gal(\cL/\cK)$, and let $\widehat H^i(H, \mathfrak{F}(m_\cL))$ denote the Tate cohomology groups.
 By \cref{trace surjective}, we know that $\widehat{H}^0(H,\mathfrak{F}(m_\cL))=0$. By \cite[Lemma 3.9]{EllerbrockNickel}, the Herbrand quotient of $\mathfrak{F}(m_\cL)$ vanishes. Thus, $\widehat{H}^i(H,\mathfrak{F}(m_\cL))=0$ for all $i$, which implies the desired claim. 
\end{proof}

{
\begin{rem} \label{remark-Hilbert}
    For determining whether an extension is weakly ramified, one may use Hilbert's formula \cite[IV§1, Prop. 4]{serre-local-fields}. This states that if $\mathcal{L}/\mathcal{K}$ is a Galois extension of local fields, and $v$ denotes the valuation on $L$, then the different $\mathfrak D_{\mathcal L/\cK}$ has valuation
    \[v(\mathfrak D_{L/K})=\sum_{i=0}^\infty (|G_i|-1),\]
    where $G_i$ are the higher ramification groups (in lower numbering). In particular, a totally ramified extension of degree $p$ resp.~$p^2$ is weakly ramified if and only if $v(\mathfrak D_{L/K})=2p-2$ resp.~$2p^2-2$.
\end{rem}
\begin{lemma} \label{lem.weak-ram}
    Let $\mathcal{L}/\mathcal{K}$ and $\mathcal{K}'/\mathcal{K}$ be Galois extensions of local fields of degree $[\mathcal L:\cK]=[\cK':\cK]=p$. Assume that $\mathcal{L}\cap\mathcal{K}'=\mathcal{K}$ and that $\mathcal{L}/\mathcal{K}$ is weakly ramified. Then $\mathcal{L}\mathcal{K}'/\mathcal{K}'$ is weakly ramified.
\end{lemma}
\begin{proof}
    If $\mathcal{L}\mathcal{K'}/\mathcal{K}$ is at most weakly ramified, there is nothing to prove. It remains to treat the case when $\mathcal{K}'\mathcal{L}/\mathcal{K}$ is totally but not weakly ramified. Let $v$ be the valuation on $\mathcal{K}'\mathcal{L}$. As $\mathcal{K}'\mathcal{L}/\mathcal{K}$ is a $(\Z/p\Z)^2$-extension that is not weakly ramified, \cref{remark-Hilbert} yields
    \[2p^2-2\neq v(\mathfrak{D}_{\mathcal{L}\mathcal{K}'/\mathcal{K}})=v(\mathfrak{D}_{\mathcal{L}\mathcal{K}'/\mathcal{L}})+v(\mathfrak{D}_{\mathcal{L}/\mathcal{K}})=v(\mathfrak{D}_{\mathcal{K}'\mathcal{L}/\mathcal{L}})+p(2p-2).\]
    It follows that $v(\mathfrak{D}_{\mathcal{K}'\mathcal{L}/\mathcal{L}})\neq 2p-2$. In particular, $\mathcal{K}'\mathcal{L}/\mathcal{L}$ is not weakly ramified by \cref{remark-Hilbert}. By \cite[Theorem 10.7]{neukirch} and the fact that $\mathcal{L}/\mathcal{K}$ is weakly ramified, we obtain
$$1=G_2(\mathcal{L}/\mathcal{K})=G_2(\mathcal{L}\mathcal{K}'/\mathcal{K})\Gal(\mathcal{L}\mathcal{K}'/\mathcal{L})/\Gal(\mathcal{L}\mathcal{K}'/\mathcal{L}).$$
    This implies $G_2(\mathcal{L}\mathcal{K}'/\mathcal{K})=\Gal(\mathcal{L}\mathcal{K}'/\mathcal{L})$. By \cite[Theorem 10.3]{neukirch} we obtain
    \[G_2(\mathcal{L}\mathcal{K}'/\mathcal{K}')=G_2(\mathcal{L}\mathcal{K}'/\mathcal{K})\cap \Gal(\mathcal{L}\mathcal{K}'/\mathcal{K}')=\Gal(\mathcal{L}\mathcal{K}'/\mathcal{L}) \cap \Gal(\mathcal{L}\mathcal{K}'/\mathcal{K}')=\{0\},\]
    as claimed.
\end{proof}}

\section{Local considerations} \label{sec:local-cons}

Let $E/\Q_p$ be an elliptic curve with good supersingular reduction and $a_p=0$, and let $\widehat{E}$ denote the formal group associated with $E$. Let $\cK/\Q_p$ be a finite extension with ramification index $e(\cK/\Q_p)$ not divisible by $p^2-1$. Let $\cK_\infty$ be the cyclotomic $\Z_p$-extension of $\cK$ and let $\cK_n$ be the intermediate field of degree $p^n$ over $\cK$. 
 
\begin{lemma} \label{lem:torsion-freeness}
    If $p^2-1\nmid e(\cK/\Q_p)$, then for all $n\ge 0$, the group $\widehat E(\cK_n)$ is $p$-torsion free. In particular, $E(\cK_n)$ is also $p$-torsion free.
\end{lemma}
\begin{proof}
    Let $\cK'$ be the unramified quadratic extension of $\cK$, and let $\cK'_n\colonequals\cK_n\cK'$. Assume that $\widehat{E}(\cK'_n)[p]\neq 0$. Recall from the proof of \cite[Proposition 8.7]{Kobayashi} that this implies that $\cK'_n$ contains all $p$-torsion points. In particular, $\cK'_n$ contains $\mathcal{L}\colonequals\cK'(\widehat{E}[p])$. 
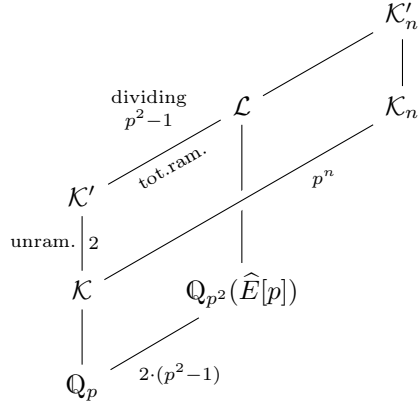
    \[\begin{tikzcd}[every arrow/.append style={no head}]
                                            &                                             & \cK'_n          \\
                                            & \cL \arrow[ru] \arrow[dd] & \cK_n \arrow[u] \\
    \cK' \arrow[ru, "\substack{\text{dividing} \\p^2-1}" near end, "\text{tot.ram.}"' {anchor=north, rotate=30, yshift=-0.5mm}]  & & \\
    \cK \arrow[u, "2"', "\text{unram.}"] \arrow[rruu, "p^n"' near end, crossing over] & \Q_{p^2}(\widehat{E}[p]) \ar[ld, "\substack{2\cdot(p^2-1)}" near end] & \\
    \Q_p \arrow[u] & & 
    \end{tikzcd}\]
    Let $\Q_{p^2}$ be the unique unramified extension of $\Q_p$ of degree $2$. 
    As $\Q_{p^2}(\widehat E[p])/\Q_{p^2}$  is a totally ramified extension of degree $p^2-1$, it follows that $p^2-1\mid e(\cK'_n/\cK')e(\cK'/\Q_p)$. The first factor is a power of $p$ and the second one is equal to $e(\cK/\Q_p)$. By assumption, $e(\cK/\Q_p)$ is not divisible by $p^2-1$, yielding a contradiction, which shows that indeed $\widehat{E}(\cK_n)[p]=0$. The $p$-torsion freeness of $E(\cK_n)$ follows by the same argument as in \cite[Lemma~3.1]{mengfai}.
\end{proof}

\begin{def1}
    For $n\ge0$, we define the two subgroups of $\widehat E(\cK_n)$:
    \[\widehat{E}^+(\cK_n)\colonequals\left\{x\in \widehat{E}(\cK_n)\mid \Tr_{\cK_n/\cK_{m+1}}(x)\in \widehat{E}(\cK_m), 2\mid m, 0\le m\le n-1\right\}\]
    \[\widehat{E}^-(\cK_n)\colonequals\left\{x\in \widehat{E}(\cK_n)\mid \Tr_{\cK_n/\cK_{m+1}}(x)\in \widehat{E}(\cK_m), 2\nmid m, 0\le m\le n-1\right\}\]
\end{def1}
 
Fix a (non-canonical) topological generator $\gamma$ of $\Gamma$, which gives rise to an isomorphism $\Lambda=\Z_p\llbracket\Gamma\rrbracket\xrightarrow{\sim}\Z_p[[X]]$ given by $\gamma\mapsto 1+X$. For $n\ge0$, let $\Phi_n$ be the $p^n$-th cyclotomic polynomial. We {define}
\[\omega_n^+(X)\colonequals X\prod_{\substack{1\le m\le n\\ m \textup{ even}}}\Phi_m(X+1), \quad \omega_n^-(X)\colonequals X\prod_{\substack{1\le m\le n\\ m\textup{ odd}}}\Phi_m(X+1).\]
Note that $\omega_n^+(X) \omega_n^-(X)=X \omega_n(X)$.

\begin{lemma}
\label{lem:finite index}
    The group $\widehat{E}^+(\cK_n)+\widehat{E}^-(\cK_n)$ has finite index in $\widehat{E}(\cK_n)$.
\end{lemma}
\begin{proof}
   By the definitions we have
    \begin{equation} \label{eq:E-pm-omega}
        \widehat{E}(\cK_n)[\omega^\pm_n]\subset \widehat{E}^\pm(\cK_n).
    \end{equation}
   
    As $\omega_n^+(X)$ and $\omega_n^-(X)/X$ are coprime in $\mathbb Z[X]$ and thus in $\mathbb Z_p[X]$, there exists a $p$-adic integer $c\in \mathbb Z_p$ and elements $a(X),b(X)\in \Z_p[X]$ such that 
    \begin{align*}
        c &= a(X)\cdot\omega_n^+(X)+b(X) \cdot\frac{\omega_n^-(X)}{X}. \\
        \intertext{Let $y\in \widehat{E}(\cK_n)$. Multiplying by $c$, we get a decomposition }
        cy &= a(X)\cdot\omega_n^+(X)\cdot y+b(X)\cdot\frac{\omega_n^-(X)}{X}\cdot y.
    \end{align*}
    Write $y^-$ resp. $y^+$ for the first resp. second term on the right hand side. By definition, we have $y^-\in \widehat{E}(\cK_n)[\omega_n^-]$ and $y^+\in \widehat{E}(\cK_n)[\omega_n^+]$. Using \eqref{eq:E-pm-omega}, we obtain
    \[cy\in \widehat{E}^+(\cK_n)+\widehat{E}^-(\cK_n).\]
    As $\widehat{E}(\cK_n)$ has finite rank, the lemma follows.
\end{proof}

\begin{lemma} \label{lem:E-pm-intersection}
    If $p^2-1\nmid e(\cK/\Q_p)$, then $\widehat{E}(\cK)=\widehat{E}^+(\cK_n)\cap \widehat{E}^-(\cK_n)$.
\end{lemma}
\begin{proof}
    We may reproduce the first half of the proof in \cite[Proposition~8.12]{Kobayashi}. Indeed, let $x\in \widehat{E}^+(\cK_n)\cap \widehat{E}^-(\cK_n)$. Distinguishing by parity, we show that for all $0\le m\le n-1$, $x\in \widehat E(\cK_m)$ implies $x\in \widehat E(\cK_{m-1})$; since $\widehat E(\cK_{-1})=\widehat E(\cK)$, this clearly implies the assertion. We have $p^{n-m} x = \Tr_{\cK_n/\cK_m}(x) \in \widehat E(\cK_{m-1})$, and thus for all $\sigma\in\Gal(\cK_m/\cK_{m-1})$ we have $p^{n-m} (\sigma(x)-x)=0$, and now \cref{lem:torsion-freeness} shows $\sigma(x)=x$, that is, $x\in \widehat E(\cK_{m-1})$.
\end{proof}
\begin{cor}
\label{ranK-finite-level}
    If $p^2-1\nmid e(\cK/\Q_p)$,  then for all $n\ge 0$ we have
    \[\Z_p\textup{-rank}(\widehat{E}^\pm(\cK_n))=[\cK:\Q_p]\deg(\omega_n^\pm).\]
\end{cor}
\begin{proof}
    We prove the claim by induction on $n$. If $n=0$, we have $\widehat{E}^-(\cK)=\widehat{E}(\cK)=\widehat{E}^+(\cK)$, which is of rank $[\cK:\Q_p]$. Let us now assume that the claim is proved for $n-1$ and that $n$ is even (the case of odd $n$ can be treated similarly). In this case $\widehat{E}^-(\cK_n)=\widehat{E}^-(\cK_{n-1})$ and $\omega_n^-=\omega_{n-1}^-$. Thus, for the minus sign there is nothing to prove. By \cref{lem:finite index,lem:E-pm-intersection}, we have
    \begin{align*}\Z_p\textup{-rank}\left(\widehat{E}^+(\cK_n)\right)&=[\cK_n:\Q_p]-\Z_p\textup{-rank}\left(\widehat E^-(\cK_n)\right)+\Z_p\textup{-rank}\left(\widehat{E}(\cK)\right)\\
    &=[\cK:\Q_p](p^n-\deg(\omega_n^-)+1)=[\cK:\Q_p]\deg(\omega_n^+),\end{align*}
    which proves the claim.
\end{proof}

\subsection{Corank computations}

Assume that $\cK/\Q_p$ is Galois, at most weakly ramified and that $\cK\cap \Q_{p,\infty}=\Q_p$. Let $\mathcal{F}\subset \cK$ be a subfield $G=\Gal(\cK_\infty/\mathcal{F}_\infty)$. 
Then $\Gal(\cK_\infty/\mathcal{F})=\Gamma\times G$. Let $\Lambda=\Z_p[[\Gamma]]$.
\begin{lemma}
\label{lower-bound}
    If $p^2-1\nmid e(\cK/\Q_p)$, then
    \[\Lambda\textup{-corank}(E^\pm(\cK_\infty)\otimes \Q_p/\Z_p)\ge [\cK:\Q_p].\]
\end{lemma}
\begin{proof}
By \cref{ranK-finite-level} we have 
    \[\Z_p\textup{-corank}(E^\pm(\cK_n)\otimes \Q_p/\Z_p)=\left(1+\sum_{\substack{k=1\\ k\textup{ even/odd}}}^{p^{n}}\varphi(p^k)\right)[\cK:\Q_p].\]
    As $E(\cK_\infty)[p]=\{0\}$, there is a natural embedding
    \[E^\pm(\cK_n)\otimes \Q_p/\Z_p \hookrightarrow (E^\pm(\cK_\infty)\otimes \Q_p/\Z_p)[\omega^\pm_n].\]

    As $\deg(\omega^\pm_{n})=1+\sum_{k=1, k\textup{ even/odd}}^{p^n}\varphi(p^k)$, we obtain
    \[\Lambda\textup{-corank}(E^\pm(\cK_\infty)\otimes \Q_p/\Z_p)\ge [\cK:\Q_p]. \qedhere\]
\end{proof}

\begin{cor}
\label{cor.rank} If $p^2-1\nmid e(\cK/\Q_p)$, then
    \[\Lambda\textup{-corank}(E^\pm(\cK_\infty)\otimes \Q_p/\Z_p)= [\cK:\Q_p].\]
\end{cor}
In the proof, we will use the following plus/minus subgroups of $H^1(\cK_n, T)$; the definition is analogous to \cite[Definition~6.1]{Kobayashi}.
\begin{def1}
    Let $H_{\pm}(\cK_n,T)\subset H^1(\cK_n, T)$ be the orthogonal complement of $\widehat{E}^\pm(\cK_n)\otimes \Q_p/\Z_p$ under the Tate pairing {$H^1(\cK_n, E[p^\infty])\times H^1(\cK_n, T)\to \Q_p/\Z_p$}, where $T=T_pE$ is the Tate module. 
\end{def1}

\begin{proof}[Proof of \cref{cor.rank}]
    We know that 
    \[\Lambda\textup{-corank}(E(\cK_\infty)\otimes \Q_p/\Z_p)= \Lambda\textup{-corank}(H^1(\cK_\infty,E[p^\infty]))=2[\cK:\Q_p],\]
    where the first equality is \cref{perrin-riou}, and the second one is due to Greenberg, see \cite[94]{Greenberg1999} or \cite[Proposition 1]{Greenberg1989-padic}.
    Note that $\widehat{E}^\pm(\cK_\infty)\otimes \Q_p/\Z_p=E^\pm(\cK_\infty)\otimes \Q_p/\Z_p$ as $E(\cK_\infty)[p]=\{0\}$ by \cref{lem:torsion-freeness}.

By \cref{lem:finite index} we see that $\widehat{E}^+(\cK_n)\otimes \Q_p/\Z_p+\widehat{E}^-(\cK_n)\otimes \Q_p/\Z_p=\widehat{E}(\cK_n)\otimes \Q_p/\Z_p$. Therefore, $H_+(\cK_n,T)\cap H_-(\cK_n,T)$ is the orthogonal complement of $\widehat{E}(\cK_n)\otimes \Q_p/\Z_p$. By definition, 
\[\varprojlim_nH_+(\cK_n,T)\cap \varprojlim_n H_-(\cK_n,T)=\varprojlim_n(H_+(\cK_n,T)\cap H_-(\cK_n,T))\]
is the orthogonal complement of $\widehat{E}(\cK_\infty)\otimes \Q_p/\Z_p$, which is in turn equal to $H^1(\cK_\infty,E[p^\infty])$ by \cref{perrin-riou}. Thus, 
\begin{equation}\label{eq_intersection}\varprojlim_nH_+(\cK_n,T)\cap \varprojlim_n H_-(\cK_n,T)=0.\end{equation}

By \cref{lem:finite index} we have the following exact sequence
\[\widehat{E}(\cK)\otimes \Q_p/\Z_p\to \widehat{E}^+(\cK_n)\otimes \Q_p/\Z_p\oplus \widehat{E}^-(\cK_n)\otimes \Q_p/\Z_p\to \widehat{E}(\cK_n)\otimes \Q_p/\Z_p\to 0,\]
which shows that 
\[\left(\widehat{E}^+(\cK_n)\otimes \Q_p/\Z_p\right)\cap \left(\widehat{E}^-(\cK_n)\otimes \Q_p/\Z_p\right)\]
is annihilated by the variable $X$ of $\Z_p[[X]]$. This implies that 
\[\Lambda\textup{-rank}\Big(\varprojlim_nH_+(\cK_n,T)+\varprojlim_nH_-(\cK_n,T)\Big)=\Lambda\textup{-rank}\Big(\varprojlim_nH^1(\cK_n,T)\Big).\]
Using \eqref{eq_intersection}, we obtain
\[\Lambda\textup{-rank}\Big(\varprojlim_nH_+(\cK_n,T)\Big)+\Lambda\textup{-rank}\Big(\varprojlim_n H_-(\cK_n,T)\Big)=\Lambda\textup{-rank}\Big(\varprojlim_n H^1(\cK_n,T)\Big)=2[\cK:\Q_p].\]
Without loss of generality we can assume that 
\[\Lambda\textup{-rank}\Big(\varprojlim_nH_+(\cK_n,T)\Big)\ge \Lambda\textup{-rank}\Big(\varprojlim_n H_-(\cK_n,T)\Big).\]

Therefore,
\begin{align*}
[\cK:\Q_p]&\le \Lambda\textup{-rank}\Big(\varprojlim_nH_+(\cK_n,T)\Big)\\
&=2 [\cK:\Q_p]-\Lambda\textup{-corank}\Big(\widehat{E}^+(\cK_\infty)\otimes \Q_p/\Z_p\Big)\le [\cK:\Q_p]. \end{align*}
As $\varprojlim_n H_+(\cK_n,T)$ is the Tate dual of $\widehat{E}^+(\cK_\infty)\otimes \Q_p/\Z_p$, we obtain 
\begin{equation}
    \label{eq:ranks} [\cK:\Q_p]=\Lambda\textup{-corank}\Big(\widehat{E}^+(\cK_\infty)\otimes \Q_p/\Z_p\Big)=\Lambda\textup{-rank}\Big(\varprojlim_nH_+(\cK_n,T)\Big).
\end{equation}
As a consequence we obtain $\Lambda\textup{-rank}(\varprojlim_nH_-(\cK_n,T))=[\cK:\Q_p]$ and
\begin{align*}[\cK:\Q_p]&= \Lambda\textup{-rank}\Big(\varprojlim_nH_-(\cK_n,T)\Big)\\&= 2[\cK:\Q_p]-\Lambda\textup{-corank}\Big(\widehat{E}^-(\cK_\infty)\otimes \Q_p/\Z_p\Big)\le [\cK:\Q_p], \end{align*}
which completes the proof.
\end{proof}

\subsection{Freeness of Iwasawa cohomology}
In this subsection, we study the Iwasawa cohomology groups $\Hiw(\cK,T)\colonequals\varprojlim H^1(\cK_n,T)$ and $\Hpm(\cK,T)\colonequals\varprojlim H_\pm(\cK_n,T)$ as modules over $\Lambda$ and $\Lambda[G]$.
\begin{lemma}
\label{free}
 The module  $\Hiw(\cK,T)=\varprojlim H^1(\cK_n,T)$ is $\Lambda$-free.
\end{lemma}
\begin{proof} Note that $(\varprojlim H^1(\cK_n,T))_\Gamma$ is the dual of $H^1(\cK_\infty,E[p^\infty])^\Gamma=H^1(\cK,E[p^\infty])$ under the Tate pairing. Thus, $(\varprojlim H^1(\cK_n,T))_\Gamma$ is $\Z_p$-free. As 
\[\Lambda\textup{-corank}(H^1(\cK_\infty,E[p^\infty])=[K:\Q_p]=\Z_p\textup{-corank}(H^1(\cK,E[p^\infty])),\]
the claim follows.
\end{proof}

\begin{lemma}
\label{hpm-free} If $p^2-1\nmid e(\cK/\Q_p)$, then
    $\Hpm(\cK,T)=\varprojlim H_\pm(\cK_n,T)$ is $\Lambda$-free.\footnote{We thank Andreas Nickel for pointing out the proof of this fact given below.}
\end{lemma}
\begin{proof}  According to \cite[\nopp 5.3.19(ii)]{NSW} it suffices to show that $\Hpm(\cK,T)^\Gamma=0$ and that $\Hpm(\cK,T)_\Gamma$ is $\Z_p$-free. 
    By \cref{free} and the result just cited, we see that $\Hpm(\cK,T)^\Gamma \subset \Hiw(\cK,T)^\Gamma=0$. It remains to prove that $\Hpm(\cK,T)_\Gamma$ is $\Z_p$-free. 

    Consider the tautological exact sequence
    \[0\to \Hpm(\cK,T) \to \Hiw(\cK,T)\to \Hiw(\cK,T)\big/\Hpm(\cK,T) \to 0. \]
   The snake lemma gives us
   \[0\to \left(\Hiw(\cK,T)/\Hpm(\cK,T)\right)^\Gamma\to \Hpm(\cK,T)_\Gamma\to \Hiw(\cK,T)_\Gamma.\]
   The last term is $\Z_p$-free by \cref{free}. It remains to show that $(\Hiw(\cK,T)/\Hpm(\cK,T))^\Gamma$ is $\Z_p$-free. Note that the {dual with respect to the Tate pairing agrees with the Pontryagin dual}. To prove the desired freeness, consider the following chain of equalities:
   \begin{align*}
       \left(\left(\Hiw(\cK,T)\big/\Hpm(\cK,T)\right)^\Gamma\right)^\vee
       &=\left(\left(\Hiw(\cK,T)\big/\Hpm(\cK,T)\right)^\vee\right)_\Gamma\\
       &=\left(\left(\varprojlim_n H^1(\cK_n,T)\big/H_\pm(\cK_n,T)\right)^\vee\right)_\Gamma\\
       &=\left(\varinjlim_n \widehat E^\pm(\cK_n,T)\otimes \Q_p/\Z_p\right)_\Gamma
       =\left(\widehat E^\pm(\cK_\infty)\otimes\Q_p/\Z_p\right)_\Gamma.
   \end{align*}
   Thus, 
   \begin{align*} \left(\Hiw(\cK,T)\big/\Hpm(\cK,T)\right)^\Gamma
        &=\left(\left(\widehat E^\pm(\cK_\infty)\otimes \Q_p/\Z_p\right)_\Gamma\right)^\vee \\
        &=\left(\left(E^\pm(\cK_\infty)\otimes \Q_p/\Z_p\right)_\Gamma\right)^\vee
        =\left(\widehat E^\pm(\cK_\infty)^\iota\otimes \Z_p\right)^\Gamma,
    \end{align*}
    where $\iota$ means that elements $\gamma\in \Gamma$ act via $\gamma^{-1}$ on $\widehat E^\pm(\cK_\infty)\otimes \Z_p$. We obtain 
    \[\left(\widehat E^\pm(\cK_\infty)^\iota\otimes \Z_p\right)^\Gamma = \widehat E^\pm(\cK)\otimes \Z_p,\]
    which is clearly $\Z_p$-free.
\end{proof}
\begin{rem}
    By definition, $(\widehat{E}^\pm (\cK_\infty)\otimes \Q_p/\Z_p)^\lor=\Hiw(\cK,T)/\Hpm(\cK,T)$. By \cref{free,hpm-free}, both modules $\Hiw(\cK,T)$ and $\Hpm(\cK,T)$ are $\Lambda$-free. This, however, does not imply that the quotient is $\Lambda$-free as well, as one can easily see from the following example.

    Consider the following two $\Lambda$-homomorphisms
    \[\psi\colon \Lambda\to \Lambda^2, \quad 1\mapsto (p,-T)\]
    and
    \[\phi\colon \Lambda^2\mapsto (p,T)\Lambda, \quad (1,0)\mapsto T,\quad (0,1)\mapsto p.\]
    It is easy to see that $\ker(\phi)=\Lambda(p,-T)=\textup{Im} (\psi)$. Thus we obtain a short exact sequence
    \[0\to \Lambda\to \Lambda^2\to (p,T)\Lambda\to 0\]
    The first two modules are clearly free over $\Lambda$, while the third one is not.
\end{rem}
\begin{cor}
    \label{lem:subersingular-local-term-div} If $p^2-1\nmid e(\cK/\Q_p)$, then
 $\left(\frac{H^1(\mathcal{K}_\infty,E[p^\infty])}{E^\pm(\mathcal{K}_\infty)\otimes \Q_p/\Z_p}\right)^{\Gamma_n}$ is $\Z_p$-divisible for all $n\ge 0$.  
\end{cor}
\begin{proof}
    It suffices to prove that $\varprojlim H_\pm (\cK_n,T)/\omega_n$ is $\Z_p$-free. By \cref{hpm-free} we know that $\varprojlim H_\pm (\cK_n,T)$ is $\Lambda$-free. Thus, taking $\Gamma_n$-coinvariants results in a $\Z_p$-free module. 
\end{proof}

\begin{lemma} \label{lem:LambdaG-freeness-rank} Let $G$ be a cyclic group of order $p$ that commutes with $\Gamma$. 
    Let $M$ be a finitely generated $\Lambda[G]$-module that is free as a $\Lambda$-module. {Let $M_{\Gamma,G}$ denote the module of coinvariants under $\Gamma$ and $G$. Suppose that $M_{\Gamma,G}$ is free as a $\Z_p$-module and $\Z_p\textup{-rank}(M_{\Gamma,G})\cdot |G| = \Lambda\textup{-rank}(M)$.\footnote{The authors thank Eva Brenner for pointing out a missing condition in an earlier version of this statement.}} Then $M$ is $\Lambda[G]$-free.
\end{lemma}
\begin{proof}
    Since $M$ is free as a $\Lambda$-module, we may choose a $\Z_p$-basis $x_1,\ldots,x_r$ of $M_{\Gamma,G}$, where $r=\Z_p\textup{-rank}(M_{\Gamma,G})$. Applying Nakayama's lemma, we obtain that these generate $M$ as a $\Lambda[G]$-module: $\langle x_1,\ldots,x_r\rangle_{\Lambda[G]}=M$. We have $\Lambda\textup{-rank}(\langle x_1,\ldots,x_r\rangle_{\Lambda[G]})\le r\cdot |G| = \Lambda\textup{-rank}(M)$ by the assumption. It follows that there can be no $\Lambda[G]$-relations between the generators $x_1,\ldots,x_r$, so they form a $\Lambda[G]$-basis of $M$.
\end{proof}

\begin{cor}\label{cor:Lambda-g} Let $G\subset\Gal(\cK/\Q_p)$ be cyclic of order $p$.
    If $p^2-1\nmid e(\cK/\Q_p)$, and if $\cK/(\cK)^G$ is  at most weakly ramified, then $\Hpm(\cK,T)$ is a free $\Lambda[G]$-module.
\end{cor}
\begin{proof}
    This follows by applying \cref{lem:LambdaG-freeness-rank} to $M\colonequals \Hpm(\cK,T)$. The module $\Hpm(\cK,T)$ is $\Lambda$-free by \cref{hpm-free}. To verify the rank condition, recall that $\Hpm(\cK,T)_{\Gamma,G} = H_\pm(\cK,T)_G$ is, by definition, the Tate dual of $H^1(\cK^G,E[p^\infty]) / (\widehat E^\pm(\cK^G)\otimes \Q_p/\Z_p)$. We show that this quotient is $\Z_p[G]$-cofree: for this, first consider the tautological exact sequence
    \[0\to E(\cK)\otimes \Q_p/\Z_p\to H^1(\cK,E[p^\infty])\to \frac{ H^1(\cK,E[p^\infty])}{E(\cK)\otimes \Q_p/\Z_p}\to0.\]
    We take $G$-invariants. {By Corollary \ref{cor-surjective}}, we have that $E(\cK)$ is $\Z_p[G]$-free, hence the first term in the resulting long exact sequence is $E(\cK^G)\otimes\Q_p/\Z_p$. The inflation-restriction sequence shows that the second term is $H^1(\cK^G, E[p^\infty])$. The fourth term is $H^1(G,E(\cK)\otimes\Q_p/\Z_p)=0$. Hence the third term is the quotient of the first two, that is, the long sequence of $G$-invariants becomes:
    \[0\to E(\cK^G)\otimes\Q_p/\Z_p \to H^1(\cK^G, E[p^\infty]) \to \frac{H^1(\cK^G, E[p^\infty])}{E(\cK^G)\otimes\Q_p/\Z_p} \to 0\]
    The middle term is divisible, hence $\Z_p$-cofree, and thus so is the quotient, as claimed.
    The rank condition now follows from \eqref{eq:ranks}.
\end{proof}

\subsection{Computation of cohomology groups}
In this subsection, we compute the cohomology groups which will be relevant for studying signed Selmer groups. Most of our computations follow along the lines of Lim's work \cite{mengfai}, with the crucial exception of \cref{MFL3.9}, the proof of which relies on the freeness of Iwasawa cohomology.

From now on, we assume that $\cK/\Q_p$ is {a Galois extension}.
Let $\mathcal{F}\subset \cK$ be a subfield such that $\cK$ is Galois over $\cF$. We set $G\colonequals\Gal(\cK_\infty/\cF_\infty)$. Then $G$ is canonically isomorphic to a subgroup of $\Gal(\cK/\cF)$, and it makes sense to consider the Galois action of (subgroups of) $G$ on $\widehat{E}(\cK)$.

Recall that a finite extension of local fields is called weakly ramified if the second ramification group vanishes. 
\begin{rem}
     The assumption $\Q_{p,\infty}\cap \cK=\Q_p$ ensures that there is a well-defined action of $H$ on $\cK_n$. If $\cK/\Q_p$ is tamely ramified, the condition $\Q_{p,\infty}\cap \cK=\Q_p$ is trivially satisfied. Cohomological triviality in the tamely ramified case was also established in \cite[Proposition~3.10]{EllerbrockNickel}.
 \end{rem}
\begin{lemma} Assume that $\cK/\Q_p$ is at most weakly ramified and that $\Q_{p,\infty}\cap \cK=\Q_p$. 
    \label{lem:cohomology-base-field} For each subgroup $H$ of  $\Gal(\cK/\Q_p)$ and each $i\ge 1$ we have
\[H^i(H,\widehat{E}(\cK_n))=\{0\}.\]
\end{lemma}

\begin{proof} Note that $\Gal(\cK/\Q_p)$ is the semidirect product of cyclic groups. It therefore suffices to consider the case that $H$ is cyclic. If the order of $H$ is coprime to $p$, the claim follows from \cite[Proposition 3.10]{EllerbrockNickel}. It remains the case that $\vert H\vert=p$. {By Lemma \ref{lem.weak-ram} the extension $\mathcal{K}_n/ \mathcal{K}_n^H$ is weakly ramified.}
As $E$ is supersingular and therefore $\widehat{E}$ is of height $2$, this follows from \cref{cor-surjective}. 
\end{proof}
Using \cref{lem:cohomology-base-field} instead of \cite[Lemma 3.2]{mengfai}, we can work along the lines of \cite[proof of Proposition 3.4]{mengfai} to prove the following:
\begin{lemma}
    \label{lem:hi-finite} Assume that $\cK/\Q_p$ is at most weakly ramified, and that $\Q_{p,\infty}\cap \cK=\Q_p$ and $p^2-1\nmid e(\cK/\Q_p)$.
    For every  subgroup of $\Gal(\cK/\Q_p)$ we have
    \[H^i(H,\widehat{E}(\cK_n)\otimes \Q_p/\Z_p)=\begin{cases}
        \widehat{E}(\cK_n^H)\otimes \Q_p/\Z_p \quad &i=0,\\
        0\quad &i>0. \qed
    \end{cases}\]
\end{lemma}
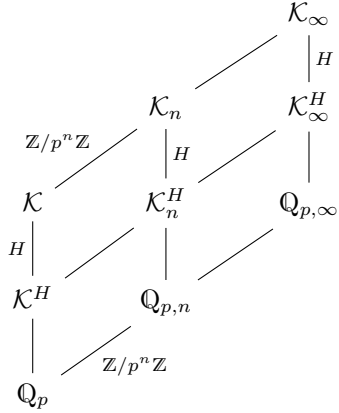
\[\begin{tikzcd}[every arrow/.append style={no head}]
                                           &                                 & \cK_\infty \arrow[d, "H"]         \\
                                           & \cK_n \arrow[ru] \arrow[d, "H"] & \cK_\infty^H \arrow[d] \arrow[ld] \\
\cK \arrow[d, "H"'] \arrow[ru, "\Z/p^n\Z"] & \cK_n^H \arrow[d]               & {\Q_{p,\infty}}                   \\
\cK^H \arrow[d] \arrow[ru]                 & {\Q_{p,n}} \arrow[ru]           &                                   \\
\Q_p \arrow[ru, "\Z/p^n\Z"']               &                                 &                                  
\end{tikzcd}\]
The infinite level version of \cref{lem:hi-finite} is the following:
\begin{lemma}\label{lem:hi}  If $p^2-1\nmid e(\cK/ \Q_p)$, then for every subgroup $H$ of {$\Gal(\cK/\Q_p)$} we have
    \[H^i(H,\widehat{E}(\cK_\infty)\otimes \Q_p/\Z_p)=\begin{cases}
        \widehat{E}(\cK_\infty^H)\otimes \Q_p/\Z_p \quad &i=0,\\
        0\quad &i>0.
    \end{cases}\] 
\end{lemma}
\begin{proof} The proof follows along the lines of \cite[Proposition 3.6]{mengfai}. We restate it here for the convenience of the reader.
      By \cref{lem:torsion-freeness}, we have a short exact sequence
     \[0\to \widehat{E}(\cK_\infty)\to \widehat{E}(\cK_\infty)\otimes \Q_p\to \widehat{E}(\cK_\infty)\otimes \Q_p/\Z_p\to 0\]
     By \cite[Theorem 3.1]{coates-greenberg} we have $H^i(H,\widehat{E}(\cK_\infty))=0$ for all $i>0$. Thus, we obtain
     \[0\to \widehat{E}(\cK_\infty^H)\to \widehat{E}(\cK_\infty^H)\otimes \Q_p\to (\widehat{E}(\cK_\infty)\otimes \Q_p/\Z_p)^H\to 0\]
     and
     \[H^i(H,\widehat{E}(\cK_\infty)\otimes \Q_p)\cong H^i(H,\widehat{E}(\cK_\infty)\otimes \Q_p/\Z_p)\]
     for $i\ge 1$. The claim for $i=0$ follows from the exact sequence. As $H$ is a finite group and $\widehat{E}(\cK_\infty)\otimes \Q_p$ is torsion free, we see that $H^i(H, \widehat{E}(\cK_\infty)\otimes\Q_p)=0$ for all $i>0$, which implies the second claim by the above isomorphism. 
\end{proof}

\begin{prop} \label{MFL3.9}
    Assume that $\cK/\Q_p$ is at most weakly ramified, that $\Q_{p,\infty}\cap \cK=\Q_p$ 
    and $p^2-1\nmid e(\cK/\Q_p)$.
    For every subgroup $H$ of {$\Gal(\cK/\Q_p)$} we have
    \[H^i\left(H,\frac{H^1(\cK_\infty,E[p^\infty])}{\widehat{E}^\pm(\cK_\infty)\otimes \Q_p/\Z_p}\right)=\begin{cases}
        \frac{H^1(\cK_\infty^H,E[p^\infty])}{\widehat{E}^\pm(\cK_\infty^H)\otimes \Q_p/\Z_p} \quad &i=0 ,\\
        0\quad &i>0.
    \end{cases} \]
\end{prop}
\begin{proof}
    We will first assume that $H$ is cyclic of prime order $\ell$. If $\ell$ is coprime to $p$, there is nothing to prove. Thus, we will assume that $H$ is cyclic of order $p$. By \cref{cor:Lambda-g}, the module $\Hpm(\cK,T)$ is $\Lambda[H]$-free. In particular, 
    \[H^i(H,\Hpm(\cK,T))=0\quad i>0.\]
    As $\Hpm(\cK,T)$ is the Tate dual of $\frac{H^1(\cK_\infty,E[p^\infty])}{\widehat{E}^\pm(\cK_\infty)\otimes \Q_p/\Z_p}$, the claim for $i>0$ follows. It remains to prove the claim for $i=0$. As $H$ is cyclic, we obtain \[\widehat{H}^0\left(H,\frac{H^1(\cK_\infty,E[p^\infty])}{\widehat{E}^\pm(\cK_\infty)\otimes \Q_p/\Z_p}\right)=H^2\left(H,\frac{H^1(\cK_\infty,E[p^\infty])}{\widehat{E}^\pm(\cK_\infty)\otimes \Q_p/\Z_p}\right)=0,\]
    where $\widehat H^0$ denotes the $0$th Tate cohomology group.
    In particular,
    \[\left(\frac{H^1(\cK_\infty,E[p^\infty])}{\widehat{E}^\pm(\cK_\infty)\otimes \Q_p/\Z_p}\right)^H=\textup{Tr}_H\left(\left(\frac{H^1(\cK_\infty,E[p^\infty])}{\widehat{E}^\pm(\cK_\infty)\otimes \Q_p/\Z_p}\right)\right)\subset \frac{H^1(\cK_\infty^H,E[p^\infty])}{\widehat{E}^\pm(\cK_\infty^H)\otimes \Q_p/\Z_p}. \]
    The last inclusion follows from the fact that \[\textup{Tr}_H (H^1(\cK_\infty,E[p^\infty]))\subset H^1(\cK_\infty,E[p^\infty])^H=H^1(\cK_\infty^H,E[p^\infty])\] and \[\textup{Tr}_H(\widehat{E}^\pm(\cK_\infty)\otimes \Q_p/\Z_p)\subset \widehat{E}^\pm(\cK_\infty)^H\otimes \Q_p/\Z_p=\widehat{E}^\pm(\cK^H_\infty)\otimes \Q_p/\Z_p.\]
    This completes the proof for cyclic groups of prime order. 

    For the general case, note that $\Gal(\cK/\Q_p)$ is solvable and that each subfield of $\cK$ is again at most weakly ramified and such that $\Q_{p,\infty}\cap \cK=\Q_p$. It therefore suffices to prove the claim for cyclic groups.
    \end{proof}
\begin{cor} \label{supersingular-herbrand}
    Assume that $\cK/\Q_p$ is at most weakly ramified, and that $\Q_{p,\infty}\cap \cK=\Q_p$ and $p^2-1\nmid e(\cK/\Q_p)$.
    For every cyclic subgroup $C=PQ$ of {$\Gal(\cK/\Q_p)$} and every character $\varepsilon$ of $Q$ we have
    \[h_P\left(\left(\frac{H^1(\cK_\infty,E[p^\infty])}{\widehat{E}^\pm(\cK_\infty)\otimes \Q_p/\Z_p}\right)^\varepsilon\right)=1. \qed\]
\end{cor}

\subsection{\texorpdfstring{$p$}{p}-primary part}

\begin{lemma}
\label{ss-p-primary}
    If $p^2-1\nmid e(\cK/\Q_p)$, then there is a natural isomorphism
    \[\frac{H^1(\cK_\infty,E[p])}{\widehat{E}^\pm(\cK_\infty)/p\widehat{E}^\pm(\cK_\infty)}\longrightarrow \frac{H^1(\cK_\infty,E[p^\infty])}{\widehat{E}^\pm (\cK_\infty)\otimes \Q_p/\Z_p}[p].\]
\end{lemma}
\begin{proof}
    Consider the following commutative diagram; the first row is tautologically exact, with the vertical arrows being the natural maps. 
    \[\begin{tikzcd}
    0 \arrow[r] & \widehat{E}^\pm(\cK_\infty)/p\widehat{E}^\pm(\cK_\infty)\arrow[r] \arrow[d] & {H^1(\cK_{\infty}, E[p])} \arrow[r] \arrow[d] &  \frac{H^1(\cK_\infty,E[p])}{\widehat{E}^\pm(\cK_\infty)/p\widehat{E}^\pm(\cK_\infty)}\arrow[d]\arrow[r]&0 \\
    0 \arrow[r] & ( \widehat{E}^\pm(\cK_\infty)\otimes \Q_p/\Z_p)[p]\arrow[r]         & {H^1(\cK_\infty, E[p^\infty])}[p] \arrow[r]                 & \frac{H^1(\cK_\infty,E[p^\infty])}{\widehat{E}^\pm (\cK_\infty)\otimes \Q_p/\Z_p}[p]\arrow[r]&0
    \end{tikzcd}\]
    {The second row is exact as $\widehat{E}^\pm \otimes \Q_p/\Z_p$ is $\Z_p$-divisible }
    The left most and the middle vertical map are isomorphisms as $E(\cK_\infty)[p]=0$ by \cref{lem:torsion-freeness}. Thus, the right vertical map is an isomorphism.
\end{proof}

\section{Global considerations} \label{sec:global-cons}
We fix the following: $p$ is a rational prime, $F/F'/\Q$ are number fields with $p$ completely split in $F'$, $K/F$ is a finite Galois extension, and $E/F'$ is an elliptic curve. We assume that 
\begin{itemize}
    \item[(S1)] $E$ has good reduction at all $p$-adic places of $F'$;
    \item[(S2)] there is a $p$-adic place with supersingular reduction;
    \item[(S3)] each $p$-adic supersingular place $u$
    \begin{enumerate}[label=\roman*)]
        \item has ramification index $e_u(K/F')$ that is not divisible by $p^2-1$ in $K/F'$,
        \item  fulfills the following condition:  $K_u$ is contained in the compositum of an at most weakly ramified extension $\cK'/\Q_p$ and the cyclotomic extension $\Q_{p,\infty}$, {where $\cK'\cap \Q_{p,\infty}=\Q_p$.}
        \item satisfies $a_u=0$.
    \end{enumerate}
\end{itemize}
Conditions (S3.i) and (S3.ii) are weaker than those put in place by \cite[\S4]{mengfai}: indeed, there the condition is $e_u(K/F')=1$, whereas we require $e_u(K/F')$ not to be divisible by $p^2-1$ and that wild ramification only come from the cyclotomic $\Z_p$-extension.

\begin{rem}
    Let $E/\Q$ be an elliptic curve and $p\ge 5$ a supersingular prime. Let $s,l$ be a non-negative integers and let $\alpha_1,\dots,\alpha_s$ be integers that are independent in $\Q^\times/(\Q^\times)^{p^l}$. Assume that the $\alpha_i$ are $p^l$-th powers in $\Q_p$. Let $K=\Q(\zeta_{p^l},\alpha_1^{1/p^l},\dots,\alpha_s^{1/p^l})$.  Then $E$, $p$ and $K$ satisfy the above conditions with $F=F'=\Q$.
\end{rem}

Fix the following finite sets of places of $F$. 
We write $\Sigma^\ord$ resp. $\Sigma^\ss$ resp. $\Sigma^{\mathrm{bad}}$ for the set of places of $F$ at which $E$ has ordinary resp. supersingular resp. bad reduction. Let $\Sigma_p$ resp. $\Sigma_\infty$ be the set of $p$-adic resp. infinite places of $F$. Let $\Sigma$ be a finite set of places of $F$ satisfying
\begin{equation} \label{eq:Sigma-conditions}
    \Sigma \supseteq \Sigma_p \cup \Sigma_{\mathrm{ram}(F/F')} \cup \Sigma_{\mathrm{ram}(K/F)} \cup \Sigma^{\mathrm{bad}} \cup \Sigma_\infty,
\end{equation}
and let $\Sigma_1 \colonequals \Sigma-\Sigma_p$ be the subset of non-$p$-adic places. 
We decompose $\Sigma_p$ into the disjoint sets $\Sigma^{\textup{ord}}_p$ of ordinary places and $\Sigma^{\ss}_p$ of supersingular places. For an extension $\widetilde F/F$, let $\Sigma(\widetilde F)$ denote the places of $\widetilde F$ above those in $\Sigma$.

Let $K_\infty/K$ by the cyclotomic $\Z_p$-extension, and let $K/L$ be a subextension of $K/F$. If $L_\infty=LF_\infty$ is the cyclotomic $\Z_p$-extension of $L$ and $L_m/L$ is the unique degree $p^m$ extension of $L$ in $L_\infty$, then there is an $m\ge0$ such that $K\cap L_\infty = L_m$.
Write $H_L\colonequals \Gal(K_\infty/L_\infty)$ and $\Gamma_L\colonequals\Gal(L_\infty/L)$. Let $G=\Gal(K_\infty/F_\infty)$ and denote $\Gamma_F$ by $\Gamma$. Then we have an isomorphism $\G\colonequals\Gal(K_\infty/F)=G\rtimes \Gamma$. {We fix once and for all a lift $\Gamma'$ of $\Gamma$ in $\mathcal{G}$ such that the restriction induces an isomorphism $\Gamma'\cong \Gamma$. By abuse of notation we will denote $\Gamma'$ by $\Gamma$ in the following. } Let $\Lambda\colonequals\Z_p\llbracket\Gamma_K\rrbracket$ and $\Lambda(\G)\colonequals\Z_p\llbracket\G\rrbracket$ be the relevant Iwasawa algebras.

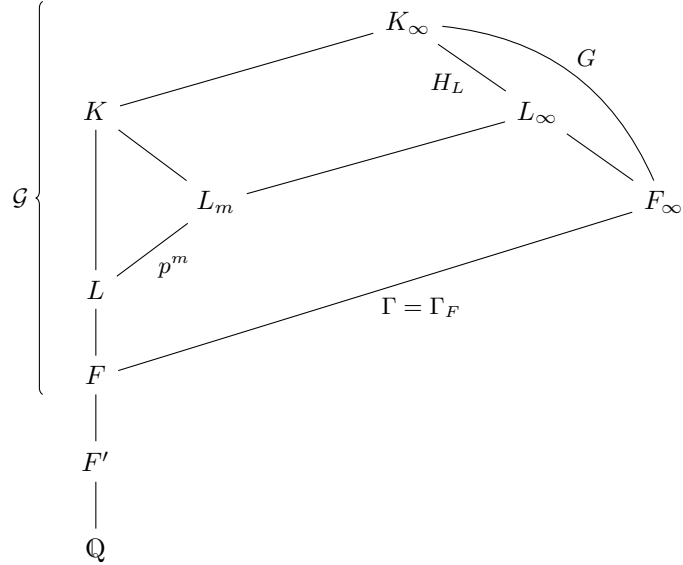
\[\begin{tikzcd}[every arrow/.append style={no head}, every label/.append style = {font = \small}]
{\phantom K}  \arrow[dddd, start anchor=north, end anchor=south, no head, xshift=-2em, decorate, decoration={brace,mirror}, "\G" left=3pt] & &  & K_\infty \arrow[rd, "H_L"'] \arrow[rrdd, "G", bend left] & & \\
K \arrow[rrru] \arrow[dd] \arrow[rd] & & & & L_\infty \arrow[rd] & \\
& L_m \arrow[rrru] & & & & F_\infty \\
L \arrow[ru, "p^m"'] \arrow[d] & & & & & \\
F \arrow[rrrrruu, "\Gamma=\Gamma_F"'] \arrow[d] & &  & & & \\
F' \arrow[d] & & & & & \\
\Q & & & & &
\end{tikzcd}\]
{Note that conditions (S2) and (S3) imply that $H^0(G_\Sigma(K_\infty,E[p^\infty])=0$ by Lemma \ref{lem:torsion-freeness}. We will frequently use this fact without further mentioning it.}
\subsection{Definition of signed Selmer groups} \label{sec:definition-of-signed-Sel}
For each supersingular place $v\in \Sigma_p^\ss$, fix a sign $s_v\in \{+,-\}$, thus defining a vector $\vec s\in \{\pm\}^{\Sigma_p^\ss}$. Let $v(L_n)$ denote the set of primes lying above $v$ in $L_n$; note that there may be more than one such prime. For $u\in v(L_n)$, let $s_u\colonequals s_v$: this defines a vector $\vec s(L_n)\in \{\pm\}^{\Sigma_p^\ss(L_n)}$.

Let $L_{n,\Sigma}/L_n$ be the maximal $\Sigma$-ramified extension of $L_n$. Note that the assumption $\Sigma\supseteq \Sigma_p \cup \Sigma_{\mathrm{ram}(K/F)}$ implies $L_\Sigma=K_\Sigma$ and $L_{\infty,\Sigma}=K_{\infty,\Sigma}$. Let the signed Selmer group $\Sel^{\vec s}(E/L_n)$ be defined as the kernel of the following natural global-to-local map:
\begin{align*}
    H^1\left(L_{n,\Sigma}/L_n, E[p^\infty]\right) &\to \bigoplus_{u\in \Sigma_p^\ord(L_n)} \frac{H^1(L_{n,u}, E[p^\infty])}{E(L_{n,u})\otimes \Q_p/\Z_p)} \,\oplus \\
    &\oplus\bigoplus_{u\in \Sigma_p^\ss(L_n)} \frac{H^1(L_{n,u}, E[p^\infty])}{E^{s_u}(L_{n,u}) \otimes\Q_p/\Z_p} \oplus\bigoplus_{u\in \Sigma_1(L_n)} H^1(L_{n,u}, E[p^\infty])
\end{align*}
Let $\Sel^{\vec s}(E/L_\infty)\colonequals\varinjlim_n \Sel^{\vec s}(E/L_n)$. For $n\le \infty$ and $u\in \Sigma(L_n)$, we introduce the notation
\begin{equation} \label{eq:def-B}
    {B_{n,u}} \colonequals \begin{cases}
    E(L_{n,u})\otimes \Q_p/\Z_p & u\in \Sigma_p^\ord(L_n) \\
    E^{s_u}(L_{n,u}) \otimes\Q_p/\Z_p & u\in \Sigma_p^\ss(L_n) \\
    0 & u\in \Sigma_1(L_n)
    \end{cases}
\end{equation}
and let $J_u(E/L_n)\colonequals H^1(L_{n,u}, E[p^\infty]) / {B_{n,u}}$, so that
\begin{equation} \label{eq:Sel-def-ker}
    \Sel^{\vec s}(E/L_n) = \ker\left(H^1\left(L_{n,\Sigma}/L_n, E[p^\infty]\right) \to \bigoplus_{u\in \Sigma(L_n)} J_u(E/L_n)\right).
\end{equation}
Note that for $u\in\Sigma_p^\ord(L_\infty)$, we have 
\begin{equation} \label{eq:J-ordinary}
    J_u(E/L_\infty)\simeq H^1(L_{\infty,u}, E)[p^\infty]
\end{equation}

\begin{lemma} \label{MFL4.1}
    The restriction map $H^1(K_{\infty,\Sigma}/L_\infty,E[p^\infty])\to H^1(K_{\infty,\Sigma}/K_\infty,E[p^\infty])$ induces a map $\Sel^{\vec s}(E/L_\infty) \to \Sel^{\vec s}(E/K_\infty)^{H_L}$, which has trivial kernel and finite cokernel.
\end{lemma}
\begin{proof}
    The idea is the same as in the proof of \cite[Lemma~3.3]{HachimoriMatsuno}.
    Reformulating \eqref{eq:Sel-def-ker}, we find that there is a commutative diagram with exact rows:
    \[\begin{tikzcd}
    0 \arrow[r] & \Sel^{\vec s}(E/L_\infty) \arrow[r] \arrow[d] & {H^1(K_{\infty,\Sigma}/L_\infty, E[p^\infty])} \arrow[r] \arrow[d, "\res"] & \displaystyle\bigoplus_{u\in\Sigma(L_\infty)} J_u(E/L_\infty) \arrow[d, "\ell=\oplus_u \ell_u"] \\
    0 \arrow[r] & \Sel^{\vec s}(E/K_\infty)^{H_L} \arrow[r]         & {H^1(K_{\infty,\Sigma}/K_\infty, E[p^\infty])^{H_L}} \arrow[r]                 & \displaystyle\bigoplus_{w\in\Sigma(K_\infty)} J_w(E/K_\infty)^{H_L}
    \end{tikzcd}\]
    This is referred to as the `fundamental diagram' in \cite[\S6.2]{LeiZerbes}.
    Consider the restriction map. By \cref{lem:torsion-freeness}, we have $E(K_{\infty,w})[p^\infty]=0$ for all $w\in \Sigma^{\ss}_p(K_\infty)$. Writing $G_\Sigma(K_\infty)\colonequals\Gal(K_{\infty,\Sigma}/K_\infty)$, the inflation--restriction exact sequence implies that 
    \[\ker(\res)=H^1\left(H_L,E[p^\infty]^{G_\Sigma(K_\infty)}\right)=0\quad \coker(\res)=H^2\left(H_L,E[p^\infty]^{G_\Sigma(K_\infty)}\right)=0.\]
    In particular, $\res$ is an isomorphism.  
    We turn to {the map} $\ell$. Local considerations show that $\ell_u$ is an isomorphism for all supersingular primes above $p$: indeed, \cref{MFL3.9} with $i=0$ shows $J_u(E/L_\infty)=J_w(E/K_\infty)^{H_L}$. For non-supersingular places, the map $\ell_u$ has finite kernel: for $u\in\Sigma_p^\ord(L_\infty)$, this is the last sentence of \cite[Lemma~3.3]{HachimoriMatsuno} together with the observation \eqref{eq:J-ordinary}, and for $u\in \Sigma_1(L_\infty)$, this is shown in loc.cit. Modding out by the cokernels on the right in both rows, the assertion follows by invoking the snake lemma.
\end{proof}

\subsection{Torsion properties of \texorpdfstring{$X^{\vec s}(E/K_\infty)$}{X\^s(E/K\_infty)}}
We want to generalise a well-known characterisation of being $\Lambda$-cotorsion for signed Selmer groups, see \cite[Proposition~4.2]{mengfai}, \cite[Proposition~2.8]{LeiLim}.

Let $n\ge0$ be fixed.
Let $T=T_pE$ denote the Tate module, and write $V= T_p\otimes\Q_p$ and $T^*(1)=\Hom_{\Z_p}(T,\Z_p(1))$. Noting that $V/T\simeq E[p^\infty]$, the Cassels--Poitou--Tate exact sequence \cite[Theorem~1.5]{CoatesSujathaGCEC} reads as follows:
\begin{align}
\label{eq:perrin-riou} \nonumber
    0&\to \Sel^{\vec s}(E/K_n)\to H^1(G_\Sigma(K_n) , E[p^\infty])\to \bigoplus_{w\in\Sigma(K_n)} J_w(E/K_n) \to \\
    &\to H^1_A(K_n, T^*(1))^\lor \to H^2(G_\Sigma(K_n), E[p^\infty]) \to \bigoplus_{w\in \Sigma(K_n)} H^2(K_{n,w}, E[p^\infty]) \to \\\nonumber
    &\to H^0(K_n, T^*(1))^\lor \to 0
\end{align}
The group $H^1_A(K_n, T^*(1))\subseteq H^1(G_\Sigma(K_n), T^*(1))$ consists of cocycles whose restrictions at $w$ satisfy local conditions $A_w\subseteq H^1(K_{n,w},T^*(1))$, where ${A_{n,w}}$ is the orthogonal complement of ${B_{n,w}}$ under local Tate duality, {where $B_{n,w}$ is as in \eqref{eq:def-B}}. As explained in the paragraph preceding $(4)$ in op.cit., Tate duality gives an isomorphism $H^2(K_{n,w},E[p^\infty])\simeq \big(\varprojlim_{m} H^0(K_{n,w}, E[p^m])\big)^\lor=0$ for all $w{\in} \Sigma(K_n)$. Going up the tower, it follows that for any place $w\in\Sigma(K_\infty)$, we have
\begin{equation} \label{eq:h2trivial}
    H^2(K_{\infty,w},E[p^\infty])=0.
\end{equation}
{This last assertion can also be seen through more direct means.\footnote{The authors thank the referee for pointing this out.} Letting $v$ denote the place beneath $w$ in $K$, the field $K_{\infty,w}$ is an extension of the local field $K_v$ of degree $p^\infty$, and therefore the absolute Galois group of $K_{\infty,w}$ has $p$-cohomological dimension $1$ by \cite[Theorem~7.1.8.(i)]{NSW}.}

The following is a generalization of \cite[Proposition 2.8]{LeiLim}.
\begin{prop} \label{Sel-H1-J-exact-sequence}
    The module $X^{\vec s}(E/K_\infty)$ is $\Lambda$-torsion if and only if $H^2(G_\Sigma(K_\infty),E[p^\infty])=0$ and the following sequence is exact:
    \begin{equation} \label{eq:Sel-H1-J-ses}
        0\to \Sel^{\vec s}(E/K_\infty)\to H^1(G_\Sigma(K_\infty),E[p^\infty])\to \bigoplus_{w\in \Sigma(K_\infty)}J_w(E/K_\infty)\to 0.
    \end{equation}
\end{prop}
\begin{proof}
Taking direct limits in \eqref{eq:perrin-riou} and using \eqref{eq:h2trivial}, we obtain an exact sequence:
    \begin{align*}
    0&\to \Sel^{\vec s}(E/K_\infty)\to H^1(G_\Sigma(K_\infty) , E[p^\infty])\to \bigoplus_{w\in\Sigma(K_\infty))} J_w(E/K_\infty) \to \\
    &\to  \Big(\varprojlim_n H^1_A(K_n, T^*(1))\Big)^\lor \to H^2(G_\Sigma(K_\infty), E[p^\infty]) \to  0.
\end{align*}
The $\Lambda$-module $\varprojlim_n H^1(G_\Sigma(K_n), T)$ is torsion-free by \cite[Lemma~2.6]{LeiLim}. As the module $\varprojlim_n H^1_A(K_n, T^*(1))$ is a submodule of $\varprojlim_n H^1(G_\Sigma(K_n), T)$ by definition, it does not contain any non-zero $\Lambda$-torsion submodule. Thus, $\varprojlim_n H^1_A(K_n, T^*(1))=0$ if {and} only if $\Lambda\textup{-corank}\left(\varprojlim_n H^1_A(K_n, T^*(1))\right)^\lor=0$.

Fix a lift of $\Gamma$ to $\Gal(K_\infty/F)$, and let $K_\infty^\Gamma$ denote its fixed field (a subfield of $K$).Note that this lift can be chosen to be $\Gamma_K$. By abuse of notation we will write $\Gamma$ for $\Gamma_K$ in the following. By \cite[Proposition~3]{Greenberg1989-padic} we have
\[\Lambda\textup{-corank}\left(H^1(G_\Sigma(K_\infty),E[p^\infty])\right)-\Lambda\textup{-corank}\left(H^2(G_\Sigma(K_\infty),E[p^\infty]\right)=[K^\Gamma_\infty: \Q] \]
If $w$ is a place coprime to $p$, \cite[Proposition 2]{Greenberg1989-padic} implies that $H^1(K_{\infty,w},E[p^\infty])$ is of $\Lambda$-corank zero. If $w\mid p$ and $p$ is an ordinary prime \cite[section 4]{HachimoriMatsuno} implies that $J_w(E/K_\infty)$ is of $\Lambda$-corank $[K^\Gamma_{\infty,w}:\Q_p]$. For $w\mid p$ a supersingular prime, \cref{cor.rank} implies that $J_w(E/K_\infty)$ has $\Lambda$-corank $[K^\Gamma_{\infty,w}:\Q_p]$. Thus we obtain
\[\Lambda\textup{-corank}\left(H^1(G_\Sigma(K_\infty),E[p^\infty)\right)-\Lambda\textup{-corank}\left(H^2(G_\Sigma(K_\infty),E[p^\infty]\right)=\Lambda\textup{-corank}\left(\bigoplus_{w\in \Sigma(K_\infty)}J_w(E/K_\infty)\right).\]

 It now follows that $\Sel^{\vec s}(E/K_\infty)$ is $\Lambda$-cotorsion if an only if $\Lambda\textup{-corank}(H^2(G_\Sigma(K_\infty),E))=\Lambda\textup{-corank}((\varprojlim_n H^1_A(K_n, T^*(1)))^\lor)=0$. By \cite[proposition 4]{Greenberg1989-padic} this is the case if and only if both modules vanish which is in turn equivalent to the vanishing of $H^2(G_\Sigma(K_\infty),E[p^\infty])$ and the validity of the short exact sequence in the statement of the proposition. 
\end{proof}

\begin{prop} \label{Hi-H-Sel-finite}
    If $X^{\vec s}(E/K_\infty)$ is $\Lambda$-torsion, then $H^i(H, \Sel^{\vec s}(E/K_\infty))$ is finite for all $H\le G$ and $i\ge1$.
\end{prop}
\begin{proof}
    We follow \cite[Proposition~4.4]{mengfai}. As before, let $H_L\le G$ be a finite subgroup with fixed field $L_\infty$. Consider the commutative diagram in the proof of \cref{MFL4.1}: the first row becomes a short exact sequence due to \cref{Sel-H1-J-exact-sequence}, and the second row can be extended to a long exact sequence:
    \[
    \begin{tikzcd}[font=\small, column sep=1em, row sep=1em]
    0 \arrow[r] & \Sel^{\vec s}(E/L_\infty) \arrow[r] \arrow[d]  & {H^1(G_\Sigma(L_\infty),E[p^\infty])} \arrow[r] \arrow[d, "\res"]  & \displaystyle\bigoplus_{u\in\Sigma(L_\infty)}J_u(E/L_\infty) \arrow[r] \arrow[d, "\ell=\oplus\ell_u"] & 0 \\
    0 \arrow[r] & \Sel^{\vec s}(E/K_\infty) \arrow[r] & {H^1(G_\Sigma(K_\infty),E[p^\infty])^{H_L}} \arrow[r]  \ar[draw=none]{d}[name=X, anchor=center]{}  & \left(\displaystyle\bigoplus_{w\in\Sigma(K_\infty)}J_w(E/K_\infty)\right)^{H_L}  \ar[rounded corners,
            to path={ -- ([xshift=4ex]\tikztostart.east)
                      |- (X.center) \tikztonodes
                      -| ([xshift=-4ex]\tikztotarget.west)
                      -- (\tikztotarget)}]{dll}[at end]{}  \\
    & {H^1\left(H_L,\Sel^{\vec s}(E/K_\infty)\right)} \arrow[r] & {H^1\left(H_L,H^1(G_\Sigma(K_\infty),E[p^\infty])\right)} \arrow[r] & H^1\left(H_L,\displaystyle\bigoplus_{w\in\Sigma(K_\infty)}J_w(E/K_\infty)\right) \ar[r]  &  \cdots
    \end{tikzcd}
    \]

    We compute the middle terms in the long exact sequence. For $i=0$, we have seen in the proof of \cref{MFL4.1} that $\res$ is an isomorphism. For $i\ge1$, we use the degeneration of the Hochschild--Serre spectral sequence. For this, we make the following observations. Firstly, \cref{Sel-H1-J-exact-sequence} implies $H^2(G_\Sigma(K_\infty),E[p^\infty])=0$.
    Furthermore, it follows from \cref{MFL4.1} that $X^{\vec s}(E/L_\infty)$ is torsion whenever $X^{\vec s}(E/K_\infty)$ is, and so \cref{Sel-H1-J-exact-sequence} applies with $L_\infty$ in place of $K_\infty$ as well, showing that $H^2(G_\Sigma(L_\infty),E[p^\infty])=0$.
    We have $H^0(G_\Sigma(K_\infty),E[p^\infty])=E(K_{\infty,\Sigma})[p^\infty]=0$ and $H^0(G_\Sigma(L_\infty),E[p^\infty])=0$: this is because locally at supersingular primes $w\in\Sigma^\ss_p(K_\infty)$ -- which exist by our assumption (S2) -- we have $E(K_{\infty,w})[p^\infty]=0$ due to \cref{lem:torsion-freeness}. The vanishing of these cohomology groups together shows degeneration, and we conclude that
    \begin{equation} \label{eq:Hi-H1-GSigma-Epinfty}
        H^i(H_L,H^1(G_\Sigma(K_\infty), E[p^\infty]))=\begin{cases}
        H^1(G_\Sigma(L_\infty), E[p^\infty]) & i=0; \\
        0 & i\ge1.
    \end{cases}
    \end{equation}

    We turn to the rightmost terms in the long exact sequence. At supersingular primes $w$, \cref{MFL3.9} shows that $\ell_u$ is surjective and $H^i(H_L, J_w(E/K_\infty))=0$ for $i\ge1$. For ordinary primes and at places away from $p$, we have that $\coker\ell_u$ and $H^i(H_L, J_w(E/K_\infty))$ are all finite by \cite[\S4]{HachimoriMatsuno}.

    The assertion about finiteness of the leftmost terms in the sequence follows.
\end{proof}
\subsection{Finite submodules of \texorpdfstring{$X^{\vec s}(E/K_\infty)$}{X\^s(E/K\_infty)}} \label{sec:finite-submodules}
For the remainder of this subsection, we assume that $\Sel^{\vec s}(E/K_\infty)$ is $\Lambda$-cotorsion.
 \begin{def1}
     Let $\Sigma_0=\Sigma_1\cup\{v\mid p, E\textup{ is ordinary at } v\}$ and let $M$ be a $G_{K}$-module. We define 
     \[\Selz(M/K_n)=\ker\left(H^1(G_\Sigma(K_n),M[p^\infty])\to \bigoplus_{v\in \Sigma_0(K_n)}H^1(K_{n,v},M)[p^\infty]\right)\]
 \end{def1}
 \begin{lemma}
 \label{corank-zero-selmer}
     Assume that $\Sel^{\vec s}(E/K_\infty)$ is $\Lambda$-cotorsion. Then we have 
     \[\Lambda\textup{-corank}(\Selz(E/K_\infty))=\Lambda\textup{-corank}\left(\bigoplus_{v\in \Sigma(K_\infty)\setminus \Sigma_0(K_\infty)}J_v(E/K_\infty)\right)\]
 \end{lemma}
\begin{proof}
    We have a tautological exact sequence
    \[0\to \Sel^{\vec s}(E/K_\infty)\to \Selz(E/K_\infty)\to\bigoplus_{v\in \Sigma(K_\infty)\setminus \Sigma_0(K_\infty)}J_v(E/K_\infty). \]
    As $\Sel^{\vec s}(E/K_\infty)$ is $\Lambda$-cotorsion by assumption, we obtain
    \begin{equation}
        \label{upper-bound-selz}
         \Lambda\textup{-corank}(\Selz(E/K_\infty))\le \Lambda\textup{-corank}\left(\bigoplus_{v\in \Sigma(K_\infty)\setminus \Sigma_0(K_\infty)}J_v(E/K_\infty)\right).
    \end{equation}
    By \cref{Sel-H1-J-exact-sequence} we know that $H^2(G_\Sigma(K_\infty),E[p^\infty])=0$. Thus, \cite[Proposition 3]{Greenberg1989-padic} implies
    \[\Lambda\textup{-corank}(H^1(G_\Sigma(K_\infty),E[p^\infty])=[K_\infty^\Gamma:\Q]. \]
    Therefore, 
    \begin{multline}
        \label{lower-bound-selz}
        \Lambda\textup{-corank}(\Selz(E/K_\infty))\ge [K_\infty^\Gamma:\Q]-\Lambda\textup{-corank}\left(\bigoplus_{v\in \Sigma_0(K_\infty)}J_v(E/K_\infty)\right)\\=\Lambda\textup{-corank}\left(\bigoplus_{v\in \Sigma(K_\infty)\setminus \Sigma_0(K_\infty)}J_v(E/K_\infty)\right).
    \end{multline}
    Combining \eqref{upper-bound-selz} with \eqref{lower-bound-selz} gives the desired claim. 
\end{proof}
Fix an isomorphism $\kappa:\Gamma\simeq 1+p\Z_p$. For any $\Gamma$-module $M$, let us write $M(t)\colonequals M\otimes\kappa^t$ for the $\Gamma$-module with $\kappa^t$-twisted $\Gamma$-action. We have $\Sel^{\vec s}(E(t)/K_\infty)=\Sel^{\vec s}(E/K_\infty)(t)$; see \cite[89]{Greenberg1999} and \cite[\S6.2]{Rubin}.
Let $d\colonequals\Lambda\textup{-corank}\left(\bigoplus_{v\in \Sigma(K_\infty)\setminus \Sigma_0(K_\infty)}J_v(E/K_\infty)\right)$. We now choose $t$ such that the following conditions are satisfied for all $n$.
\begin{align}
\label{eq1}
    \Z_p\textup{-corank}\left((\Selz(E/K_\infty)\otimes \kappa^t)^{\Gamma_n}\right) &= dp^n\\
    \Z_p\textup{-corank}\left((\Sel^{\vec s}(E/K_\infty)\otimes \kappa^t)^{\Gamma_n}\right) &= 0\\
    \Z_p\textup{-corank}\left(\left(\bigoplus_{v\in \Sigma_0(K_\infty)}J_v(E/K_\infty)\otimes \kappa^t\right)^{\Gamma_n}\right) &= ([K_\infty^\Gamma:\Q]-d)p^n\\\label{rank-h1}
    \Z_p\textup{-corank}\left(H^1(G_\Sigma(K_\infty),E[p^\infty])\otimes \kappa^t)^{\Gamma_n}\right)&=[K_\infty^\Gamma:\Q]p^n.
\end{align}
{The compact $\Lambda$-module $(\Sel^{\Sigma_0}(E/K_\infty))^\vee$ is pseudo-isomorphic to $\Lambda^d\oplus W$, where $W$ is a torsion $\Lambda$-module. For all but finitely many $t$ the quotients $(W(\kappa^t))_{\Gamma_n}$ are finite for all $n$. Thus all but finitely many choices of $t$ satisfy \eqref{eq1}. Similarly, the other conditions are satisfied for all but finitely many $t$.  }

\begin{lemma} \label{lem:finiite-cokernel-selz}We have an exact sequence
\[0\to \Selz(E(t)/K_n)\to H^1(G_\Sigma(K_n),E(t)[p^\infty])\xrightarrow{\varphi_n} \bigoplus_{v\in \Sigma_0(K_n)} J_v(E(t)/K_{n,v})\]
where the rightmost map has finite cokernel.
\end{lemma}
\begin{proof}
    Note that $H^0(G_\Sigma(K_\infty),E(t)[p^\infty])=H^0(G_\Sigma(K_\infty),E[p^\infty])=0$. Thus, the inflation--restriction exact sequence implies that we have an isomorphism
    \[H^1(G_\Sigma(K_n),E(t)[p^\infty])\cong H^1(G_\Sigma(K_\infty),E(t)[p^\infty])^{\Gamma_n}.\]
    It follows that we have a natural embedding $\Selz(E(t)/K_n)\to {\Selz(E(t)/K_\infty)^{\Gamma_n}}$, which in turn implies
    \[\Z_p\textup{-corank}\left(\Selz(E(t)/K_n)\right)\le dp^n.\]
   we obtain
    \begin{align*}
       \Z_p\textup{-corank}(\coker(\varphi_n))\le ([K_\infty^\Gamma:\Q]-d)p^n-[K_\infty^\Gamma:\Q]p^n+dp^n=0.
    \end{align*}
    This is only possible if $\coker(\varphi_n)$ is finite. 
\end{proof}
Consider now the natural maps
\begin{align*}
    \alpha_n\colon H^1(G_\Sigma(K_n),E(t)[p^\infty])&\to \bigoplus_{v\in \Sigma_0(K_n)}H^1(K_{n,v},E(t)[p^\infty])\\
    \beta_n \colon H^1(G_\Sigma(K_n),{T(-t)})&\to \bigoplus_{v\in \Sigma_0(K_n)}H^1(K_{n,v},{T(-t)}),
\end{align*}
where $T$ is the Tate module of $E[p^\infty]$. Note that the images of $\alpha$ and $\beta$ are orthogonal complements of each other by global Poitou--Tate duality. Let $G_{\alpha_n} \subset \prod_{v\in \Sigma_0(K_n)}H^1(K_{n,v},E(t)[p^\infty])$ be the smallest subgroup containing $\textup{Im}(\alpha_n)$ and $\prod_{v\in \Sigma_0(K_n)}E(t)(K_{n,v})\otimes \Q_p/\Z_p$. Let $G_{\beta_n}\subset \prod_{v\in \Sigma_0(K_n)}H^1(K_{n,v},{T(-t)})$ be the orthogonal complement of $G_{\alpha_n}$.
\begin{lemma}
\label{lem:Gbeta1}
    $G_{\beta_n}$ lies in $\textup{Im}(\beta_n)$, and $\beta_n^{-1}(G_{\beta_n})$ is finite.
\end{lemma}
\begin{proof}
    The orthogonal complement of $G_{\alpha_n}$ is contained in the orthogonal complement of $\textup{Im}(\alpha_n)$, which is $\textup{Im}(\beta_n)$. This proves the first claim. By \cref{lem:finiite-cokernel-selz}, $G_{\alpha_n}$ has finite index in the direct product $\prod_{v\in \Sigma_0(K_n)}H^1(K_{n,v},E(t)[p^\infty])$. Thus, $G_{\beta_n}$ is finite. By global Poitou--Tate duality, the kernel of $\beta_n$ is isomorphic to
    \[\ker\left (H^2(G_\Sigma(K_n),E(t)[p^\infty])\to \bigoplus_{v\in \Sigma_0(K_n)}H^2(K_{n,v},E(t)[p^\infty])\right).\]
    By \eqref{rank-h1} and \cite[equation (29)]{Greenberg1989-padic}, 
    \begin{equation} \label{eq:H2-finite}
        |H^2(G_\Sigma(K_n),E(t)[p^\infty])|<\infty.
    \end{equation}
    Thus, the kernel of $\beta_n$ is finite and we obtain that $\beta_n^{-1}(G_{\beta_n})$ is finite. 
\end{proof}
\begin{lemma}
\label{lem:Gbeta2}
    $\beta_n^{-1}(G_{\beta_n})=0$.
\end{lemma}
\begin{proof}
    By \cref{lem:Gbeta1} we know that $\beta_n^{-1}(G_{\beta_n})\subset H^1(G_\Sigma(K_n),{T(t)})_{\textup{tors}}$. We repeat a standard argument from \cite{GreenbergVatsal}: Consider the tautological exact sequence
    \[0\to {T(-t)}\to {T(-t)}\otimes \Q_p\to {E(-t)}[p^\infty]\to 0.\]
    Taking $G_\Sigma(K_n)$ cohomology and using that $H^0(G_\Sigma(K_n),{E(-t)}[p^\infty])\subset H^0(G_{\Sigma},E[p^\infty])=0$, we see that $H^1(G_\Sigma(K_n),{T(-t)})$ embeds into $H^1(G_\Sigma(K_n),{T(-t)}\otimes \Q_p)$, which is torsion-free. Therefore, $\beta_n^{-1}(G_{\beta_n})$ has to be trivial. 
\end{proof}
\begin{cor}
\label{cor:phi_n-surjective}
    The map $\varphi_n$ {defined in \cref{lem:finiite-cokernel-selz}} is surjective.
\end{cor}
\begin{proof}
    By \cref{lem:Gbeta2} we know that $\beta_n^{-1}(G_{\beta_n})$ is trivial. By \cref{lem:Gbeta1} $G_{\beta_n}\subset \textup{Im}(\beta_n)$. Thus, $G_{\beta_n}$ has to be trivial, hence its complement is $G_{\alpha_n}=\prod_{v\in \Sigma_0(K_n)}H^1(K_{n,v},E(t)[p^\infty])$, which in turn implies that $\varphi_n$ has to be surjective by definition.
\end{proof}
\begin{prop} \label{cokernel-of-psi-cofree}
    Assume that $\Sel^{\vec s}(E/K_\infty)$ is $\Lambda$-cotorsion.
    Let $\psi_n^{\vec s}\colon H^1(G_\Sigma(K_\infty),E(t)[p^\infty])^{\Gamma_n}\to \left(\prod_{v\in \Sigma(K_\infty)}J_v(E(t)/K_\infty)\right)^{\Gamma_n}$. Then the cokernel of $\psi_n^{\vec s}$ is $\Z_p$-cofree for all $n$ large enough.
\end{prop}
\begin{proof}
    The inflation--restriction exact sequence together with \cref{cor:phi_n-surjective} implies that we have a surjection
    \[\varphi^{\Gamma_n}_n\colon H^1(G_\Sigma(K_\infty),E(t)[p^\infty])^{\Gamma_n}\to \left(\bigoplus_{v\in \Sigma_0(K_\infty)}J_v(E(t)/K_\infty)\right)^{\Gamma_n}.\]
    {Indeed, for $n\gg 0$ we have 
    \begin{align*}
        H^1(G_\Sigma(K_\infty),E[p^\infty])&=H^1(G_\Sigma(K_n),E[p^\infty])^{\Gamma_n}\\
        H^1(K_{n,v},E[p^\infty]) &\to \bigoplus_{w\mid v}(H^1(K_{\infty,n},E[p^\infty]))^{\Gamma_n} \quad \forall v\in \Sigma_0\\
       \widehat{E}(K_{n,v})\otimes \Q_p/\Z_p&=(\bigoplus_{w\mid v}K_{\infty,v}\otimes \Q_p/\Z_p))^{\Gamma_n} \quad \forall v\in \Sigma_0\setminus \Sigma_1
    \end{align*}
    We obtain the following commutative diagram}
    \[\begin{tikzcd}[font=\small, column sep=1em, row sep=1em]
        H^1(G_\Sigma(K_n),E[p^\infty])\arrow[rr,"\varphi_n"]\arrow[d]& & \bigoplus_{v\in \Sigma_0(K_n)} J_v(E(t)/K_{n,v})\arrow[d]\\
        H^1(G_\Sigma(K_\infty),E[p^\infty])^{\Gamma_n}\arrow[rr,"\varphi_n^{\Gamma_n}"]& &\left(\bigoplus_{v\in \Sigma_0(K_\infty)}J_v(E(t)/K_\infty)\right)^{\Gamma_n}
    \end{tikzcd}
    \] {The two vertical arrows and the upper horizontal one are surjective. Thus, the lower vertical map is surjective. }
    We have a tautological exact sequence 
    \[0\to\left(\bigoplus_{v\in \Sigma(K_\infty)\setminus \Sigma_0(K_\infty)}J_v(E(t)/K_\infty)\right)^{\Gamma_n} \to \left(\bigoplus_{v\in \Sigma(K_\infty)}J_v(E(t)/K_\infty)\right)^{\Gamma_n} \to \left(\bigoplus_{v\in \Sigma_0(K_\infty)}J_v(E(t)/K_\infty)\right)^{\Gamma_n}\to 0.\]
    By \cref{lem:subersingular-local-term-div}, the leftmost term is divisible for all $n$. As $\varphi^{\Gamma_n}_n$ is surjective, {we obtain a surjection
\[\left(\bigoplus_{v\in \Sigma(K_\infty)\setminus \Sigma_0(K_\infty)}J_v(E(t)/K_\infty)\right)^{\Gamma_n}\to  \coker(\psi_n^{\vec s}).\]
As the image of a divisible module is divisible}, we obtain a surjection
    
    \[\left(\bigoplus_{v\in \Sigma(K_\infty)}J_v(E(t)/K_\infty)\right)^{\Gamma_n}_{\textup{div}}\to \coker(\psi_n^{\vec s}),\]
    where $M_{\textup{div}}$ denotes the maximal divisible submodule. 
    As the quotient of divisible submodules is again divisible, the desired claim follows. 
\end{proof}
\begin{prop} \label{no-nonzero-finite-submodules}
    If $X^{\vec s}(E/K_\infty)$ is torsion over $\Lambda$, then it has no nonzero finite $\Lambda$-submodules.
\end{prop}
\begin{proof}
    Let $t$ be as above. According to \cref{cokernel-of-psi-cofree}, there is an exact sequence
    \[0 \to \left(\Sel^{\vec s}(E(t)/K_\infty)\right)^{\Gamma_n} \to {\left(H^1(G_\Sigma(K_\infty),E(t)[p^\infty])\right)^{\Gamma_n}} \xrightarrow{\psi^{\vec s}_n} \left(\displaystyle\bigoplus_{v\in\Sigma(K_\infty)} J_v(E(t)/K_\infty)\right)^{\Gamma_n}\]
    with $\Z_p$-cofree cokernel on the right for every $n$. For each $n$, the long exact cohomology sequence associated with taking $\Gamma_n$-invariants of the $t$-twist of the short exact sequence in \cref{Sel-H1-J-exact-sequence} shows that in particular, the following sequence is exact:
    \[0\to \coker(\psi^{\vec s}_n)\to H^1\left(\Gamma_n, \Sel^{\vec s}(E(t)/K_\infty)\right) \to H^1\left(\Gamma_n, H^1(G_\Sigma(K_\infty), {E(t)}[p^\infty])\right).\]
    The group on the right vanishes. Indeed, following Ahmed and Lim \cite[Lemma~2.5]{AhmedMFL2019}, we have that the Hochschild--Serre spectral sequence $H^i(\Gamma_n,H^j(G_\Sigma(K_\infty), E(t)[p^\infty]))\Rightarrow H^{i+j}(G_\Sigma(K_n), E(t)[p^\infty])$ implies, by the fact that $\cd_p\Gamma_n=1$, that the group in question embeds into $H^2(G_\Sigma(K_n),E(t)[p^\infty])$, which is finite by \eqref{eq:H2-finite}. Then as in \cite[114]{Greenberg1989-padic}, we have that $\cd_p G_\Sigma(K_n)=2$ implies that $H^2(G_\Sigma(K_n),E(t)[p^\infty])$ is also divisible, and thus it must be trivial.

    In particular, we have that $H^1\left(\Gamma_n, \Sel^{\vec s}(E(t)/K_\infty)\right) $ is $\Z_p$-cofree, so its {Pontryagin} dual $H^0\left(\Gamma_n, \Sel^{\vec s}(E(t)/K_\infty)^\lor\right) $ is $\Z_p$-free, which is equivalent to $\Sel^{\vec s}(E(t)/K_\infty)^\lor$ having no nonzero finite $\Lambda$-submodules by \cite[Proposition~5.3.19(i)]{NSW}. Therefore $\Sel^{\vec s}(E/K_\infty)^\vee=X^{\vec s}(E/K_\infty)$ has no nonzero finite $\Lambda$-submodules either.
\end{proof}
\begin{rem}
    The proof of \cref{no-nonzero-finite-submodules} follows that of \cite[Proposition~4.5]{mengfai}, which in turn relies on ideas of Greenberg \cite[Proposition~4.14]{Greenberg1999}. The salient difference in our approach is that the argument in \cite{mengfai} uses an argument involving a local norm compatible sequence and plus/minus Coleman maps constructed in the unramified case by Kim \cite[Lemma~3.9ff.]{BDK2013}, an analogue of which is not known for ramified extensions. 
    This necessitates establishing cofreeness in another way, namely through \cref{cokernel-of-psi-cofree}.
\end{rem}

\subsection{Non-primitive Selmer groups} \label{sec:non-primitive-Sel}
Let $\Sigma'\subseteq \Sigma_1$ be a subset of the non-$p$-adic places in $\Sigma$. As in \cref{sec:definition-of-signed-Sel}, let $H_L\le G$ be a finite subgroup with fixed field $L_\infty$.
We define the non-primitive signed Selmer group with respect to $\Sigma'$ over $L_n$ as
\[\Sel^{\vec s}_{\Sigma'}(E/L_n)\colonequals \ker\left(H^1(G_\Sigma(L_n), E[p^\infty]) \to \bigoplus_{w\in (\Sigma-\Sigma')(L_n)} J_w(E/L_n)\right).\]
On infinite level, we set $\Sel^{\vec s}_{\Sigma'}(E/L_\infty)\colonequals\varinjlim_n \Sel^{\vec s}_{\Sigma'}(E/L_n)$, so that the previous formula holds for all $n\le\infty$. We write $X^{\vec s}_{\Sigma'}(E/L_\infty)$ for the Pontryagin dual of $\Sel^{\vec s}_{\Sigma'}(E/L_\infty)$.

For a finitely generated $\Lambda$-module $M$, let $\theta(M)$ denote the maximal $p$-exponent in the decomposition into elementary modules. Following \cite[Proposition~4.6]{mengfai}, we have that non-primitive signed Selmer groups inherit the following properties of signed Selmer groups:
\begin{prop}
\label{pro-4.6}
    Assume that $\Sel^{\vec s}(E/K_\infty)$ is $\Lambda$-cotorsion.
    \begin{enumerate}[label=(\roman*)]
        \item \label{4.6.a1} $0\to \Sel^{\vec s}(E/K_\infty)\to \Sel^{\vec s}_{\Sigma'}(E/K_\infty)\to \displaystyle\bigoplus_{w\in \Sigma'(K_\infty)} H^1(K_{\infty,w}, E[p^\infty])\to 0$.
        \item \label{4.6.a2} $0\to \Sel^{\vec s}_{\Sigma'}(E/K_\infty) \to H^1(G_\Sigma(K_\infty), E[p^\infty])\to \displaystyle\bigoplus_{w\in (\Sigma-\Sigma')(K_\infty)} J_w(E/K_\infty)\to 0$.
        \item \label{4.6.b} $\theta(X^{\vec s}_{\Sigma'}(E/K_\infty)) = \theta(X^{\vec s}(E/K_\infty))$.
        \item \label{4.6.c} For every subgroup $H\le G$ and $i\ge1$, the group $H^i\left(H,\Sel^{\vec s}_{\Sigma'}(E/K_\infty)\right)$ is finite.
        \item \label{4.6.d} $X^{\vec s}_{\Sigma'}(E/K_\infty)$ has no nontrivial finite $\Lambda$-submodules.
    \end{enumerate}
\end{prop}
\begin{proof}
    We have the following commutative diagram with exact rows; the top row is \cref{Sel-H1-J-exact-sequence}, and the bottom row is the definition of the non-primitive Selmer group.
    \[\begin{tikzcd}
    0 \arrow[r] & \Sel^{\vec s}(E/K_\infty) \arrow[r] \arrow[d] & {H^1(G_{\Sigma}(K_\infty),E[p^\infty])} \arrow[r, "\psi^{\vec s}"] \arrow[d, Rightarrow, no head] & {\displaystyle\bigoplus_{w\in \Sigma(K_\infty)} H^1(K_{\infty,w}, E[p^\infty])} \arrow[r] \arrow[d, two heads] & 0 \\
    0 \arrow[r] & \Sel^{\vec s}_{\Sigma'}(E/K_\infty) \arrow[r] & {H^1(G_{\Sigma}(K_\infty),E[p^\infty])} \arrow[r, "\psi^{\vec s}_{\Sigma'}"]  & {\displaystyle\bigoplus_{w\in (\Sigma-\Sigma')(K_\infty)} H^1(K_{\infty,w}, E[p^\infty])}           &  
    \end{tikzcd}\]
    The sequence \eqref{4.6.a1} follows from the snake lemma. Sequence \eqref{4.6.a2} is equivalent to surjectivity of $\psi^{\vec s}_{\Sigma'}$, which follows from the surjectivity of the other three maps in the right square and commutativity.

    Equation \eqref{4.6.b} has the same proof as in \cite{mengfai}.

    For \eqref{4.6.c}, we take $H$-invariants of \eqref{4.6.a1}, thus obtaining a long exact sequence
    \[\ldots\to H^i\left(H, \Sel^{\vec s}(E/K_\infty)\right) \to H^i\left(H,\Sel^{\vec s}_{\Sigma'}(E/K_\infty)\right) \to H^i\left(H,\displaystyle\bigoplus_{w\in \Sigma'(K_\infty)} H^1(K_{\infty,w}, E[p^\infty])\right)\to \ldots\]
    The first term here is finite by \cref{Hi-H-Sel-finite}, and the third one is also finite by \cite[\S4]{HachimoriMatsuno} (recall that $\Sigma'$ only contains non-$p$-adic places). The assertion follows.

    Let $t\in \Z$ satisfy the conditions in \cref{sec:finite-submodules}, and let $n$ be large enough as in the proof of \cref{no-nonzero-finite-submodules}. Then \eqref{4.6.a1} is applicable to the twist $E(t)$ of the elliptic curve $E$, and the long exact sequence associated with taking $\Gamma_n$-invariants reads
    \[\ldots\to H^i\left(\Gamma_n, \Sel^{\vec s}(E/K_\infty)\right) \to H^i\left(\Gamma_n,\Sel^{\vec s}_{\Sigma'}(E/K_\infty)\right) \to H^i\left(\Gamma_n,\displaystyle\bigoplus_{w\in \Sigma'(K_\infty)} H^1(K_{\infty,w}, E[p^\infty])\right)\to \ldots\]
    The first term is $\Z_p$-cofree, as seen in the last paragraph of the proof of \cref{no-nonzero-finite-submodules}, and the third term vanishes. (By the Hochschild-Serre spectral sequence,  the group in question embeds into $H^2(K_v,E[p^\infty])$ which is trivial by \cite[I.3.4]{milne2006}). Hence the middle term is also $\Z_p$-cofree, and the same argument as used in the end of the proof of \cref{no-nonzero-finite-submodules} shows the claim {(v)}.
\end{proof}
\section{Projectivity results} \label{sec:projectivity}
Let 
\[\Phi\colonequals\{v\in \Sigma_1\mid \textup{the inertia degree of $v$ in $K/F$ is divisible by $p$}\}.\]
\begin{thm}
Assume that $\Phi\subset \Sigma'$, and that $X^{\vec s}(E/K_\infty)$ is $\Lambda$-torsion with $\theta(X^{\vec s}(E/K_\infty))\le 1$. Then $X_{\Sigma'}^{\vec s}(E/K_\infty)/X^{\vec s}_{\Sigma'}(E/K_\infty)[p]$ is quasi-projective as $\Z_p[G]$-module. 
\end{thm}
\begin{proof}
This proof follows \cite[Theorem 4.7]{mengfai}. By \cref{pro-4.6}\ref{4.6.c} $H^i(H,\Sel_{\Sigma'}^{\vec s}(E/K_\infty))$ is finite for all subgroups $H\le G$ and for all $i\ge 1$.
     In particular, let $C=PQ$ be a cyclic subgroup of $G$, and $\varepsilon$ a character of $Q$. By \cref{pro-4.6}\ref{4.6.a2} and the fact that $\vert Q\vert $ is coprime to $p$, we have a short exact sequence
    \[0\to \Sel^{\vec s}_{\Sigma'}(E/K_\infty)^{\varepsilon}\to H^1(G_\Sigma(K_\infty),E[p^\infty])^\varepsilon\to\left( \bigoplus_{v\in (\Sigma-\Sigma')(K_\infty)}J_w(E/K_\infty)\right)^\varepsilon\to 0.\]
    For the middle term, using \cite[3.2 A and 3.2 B]{Greenberg1999}, we have
    \[h_P(H^1(G_\Sigma(K_\infty),E[p^\infty])^\varepsilon)=h_P\left(\left( \bigoplus_{v\in (\Sigma_1\cup\Sigma_{\ord}-\Sigma')(K_\infty)}J_w(E/K_\infty)\right)^\varepsilon\right)=1.\]
    For the rightmost term, \cref{supersingular-herbrand} shows
    \[h_P\left(\left(\bigoplus_{v\in \Sigma_{\ss}(K_\infty)}J_v(E/K_\infty)\right)^\varepsilon\right)=1\]
    Consequently,
    \[h_P(\Sel_{\Sigma'}^{\vec s}(E/K_\infty)^\varepsilon)=1.\]
    The claim now follows from \cref{criterion-quasi-projective}. 
\end{proof}

We recall the following cohomological criterion of Greenberg \cite[Proposition~2.4.1]{Greenberg2011} for checking whether an Iwasawa module admits a free resolution of length $1$. See also \cite[\S4]{NichiforPalvannan} and \cite[Proposition~2.14]{mengfai}.
\begin{prop} \label{Greenbergs-criterion}
    Let $Y$ be a finitely generated $\Lambda(\G)$-module that is torsion over $\Lambda$ and contains no nonzero finite $\Lambda$-submodules. Then $Y$ admits a free resolution of length $1$ of $\Lambda(\G)$-modules if for all subgroups $H\le G$, the cohomology groups $H^1(H,Y^\lor)$ and $H^2(H,Y^\lor)$ vanish.
\end{prop}

\begin{prop} \label{prop:free-resolution}
    Assume that $\Phi\subset\Sigma'$, that $X^{\vec s}(E/K_\infty)$ is $\Lambda$-torsion, and that every ordinary $p$-adic place $v\in\Sigma^{\ord}_p$ is either non-anomalous (i.e. if $w\mid v$ for $w$ a place of $K$, then $p\nmid|\tilde E(k_w)|$) or ramifies tamely in $K/F$. Then $X^{\vec s}_{\Sigma'}(E/K_\infty)$ admits a free resolution of $\Lambda(\G)$-modules of length $1$.
\end{prop}
\begin{proof}
    The proof consists of verifying the condition in Greenberg's criterion in the same fashion as in \cite[Theorem~4.8]{mengfai}. Let $H_L\le G$ be a finite subgroup with fixed field $L_\infty$. As in the proof of \cref{Hi-H-Sel-finite}, the short exact sequence of \cref{Sel-H1-J-exact-sequence} induces a commutative diagram with exact rows:
    \[
    \begin{tikzcd}[font=\small, column sep=1em, row sep=1em]
    0 \arrow[r] & \Sel^{\vec s}_{\Sigma'}(E/L_\infty) \arrow[r] \arrow[d] & {H^1(G_\Sigma(L_\infty),E[p^\infty])} \arrow[r] \arrow[d, "\res"]  & \displaystyle\bigoplus_{u\in(\Sigma-\Sigma')(L_\infty)}J_u(E/L_\infty) \arrow[r] \arrow[d, "\ell=\oplus\ell_u"] & 0 \\
    0 \arrow[r] & \Sel^{\vec s}_{\Sigma'}(E/K_\infty)^{H_L} \arrow[r] & {H^1(G_\Sigma(K_\infty),E[p^\infty])^{H_L}} \arrow[r]  \ar[draw=none]{d}[name=X, anchor=center]{}  & \left(\displaystyle\bigoplus_{w\in(\Sigma-\Sigma')(K_\infty)}J_w(E/K_\infty)\right)^{H_L}  \ar[rounded corners,
            to path={ -- ([xshift=4ex]\tikztostart.east)
                      |- (X.center) \tikztonodes
                      -| ([xshift=-4ex]\tikztotarget.west)
                      -- (\tikztotarget)}]{dll}[at end]{}  \\
    & {H^1\left(H_L,\Sel^{\vec s}_{\Sigma'}(E/K_\infty)\right)} \arrow[r] & 0 \arrow[r] \ar[draw=none]{d}[name=Y, anchor=center]{} & H^1\left(H_L,\displaystyle\bigoplus_{w\in(\Sigma-\Sigma')(K_\infty)}J_w(E/K_\infty)\right) \ar[rounded corners,
            to path={ -- ([xshift=4ex]\tikztostart.east)
                      |- (Y.center) \tikztonodes
                      -| ([xshift=-4ex]\tikztotarget.west)
                      -- (\tikztotarget)}]{dll}[at end]{} \\
    & {H^2\left(H_L,\Sel^{\vec s}_{\Sigma'}(E/K_\infty)\right)} \arrow[r] & 0
    \end{tikzcd}
    \]
    The middle terms in the long exact sequence vanish by \eqref{eq:Hi-H1-GSigma-Epinfty}. 

    We have $H^1\left(H_L,J_w(E/K_\infty)\right)=0$: for good ordinary $p$-adic places and for non-$p$-adic places, this is \cite[Proposition~3.1.1]{Greenberg2011}, and for supersingular $p$-adic places, this was shown in \cref{MFL3.9}. Hence $H^2\left(H_L,\Sel^{\vec s}_{\Sigma'}(E/K_\infty)\right)=0$. The cited statements also show that $\ell_u$ is surjective for all $u$, and thus $0=\coker (\ell)$. Commutativity of the diagram shows $\coker(\ell)=H^1\left(H_L,\Sel^{\vec s}_{\Sigma'}(E/K_\infty)\right)$. Hence Greenberg's criterion applies.
\end{proof}

\section{Kida's formula} \label{sec:Kida}
Let $L$ be a number field contained in $K$ such that the extensions $L/F$ and $K/F$ satisfy conditions (S1)-(S3) of \cref{sec:global-cons}. Assume furthermore that $K/L$ is Galois and that $\Gal(K/L)$ is a $p$-group. Let $\Sigma'$ be the set of all places $v\nmid p$ in $L$ such that the ramification index of $v$ in $K/L$ is divisible by $p$.
\begin{lemma}
    Let $v\in \Sigma'$. Then $\mu_p\subset L_v$. 
\end{lemma}
\begin{proof}
   The proof is the same as \cite[Lemma 5.1]{mengfai}.
\end{proof}
\begin{prop}\label{kida-for-non-primitive} Assume that $\Sel^{\vec s}(E/K_\infty)$ is $\Lambda$-cotorsion. 
    Assume that $\theta(X_{\Sigma'}^{\vec s}(E/K_\infty))\le 1$. Then we have
    \begin{align*}
        \lambda(X_{\Sigma'}^{\vec s}(E/K_\infty))&=[K_\infty:L_\infty]\cdot\lambda(X_{\Sigma'}^{\vec s}(E/L_\infty))\\
        \mu(X_{\Sigma'}^{\vec s}(E/K_\infty))&=[K_\infty:L_\infty]\cdot\mu(X_{\Sigma'}^{\vec s}(E/L_\infty)).
    \end{align*}
\end{prop}\begin{proof}
The proof of this proposition is analogous to \cite[Theorem 2.1]{hachimori-sharifi}. Note that it suffices to prove the proposition for the case that $[K_\infty:L_\infty]=p$. \Cref{pro-4.6}\ref{4.6.a1} implies that $\Sel_{\Sigma'}^{\vec s}(E/K_\infty)$ is $\Lambda$-cotorsion.
As before, let $H_L\le G$ be a finite subgroup with fixed field $L_\infty$, and consider the following commutative diagram:
\[
    \begin{tikzcd}
    0 \arrow[r] & \Sel_{\Sigma'}^{\vec s}(E/L_\infty) \arrow[r] \arrow[d]           & {H^1(G_\Sigma(L_\infty),E[p^\infty])} \arrow[r] \arrow[d, "\res"]  & \displaystyle\bigoplus_{u\in(\Sigma-\Sigma')(L_\infty)}J_u(E/L_\infty) \arrow[r] \arrow[d, "\ell=\oplus\ell_u"] & 0 \\
    0 \arrow[r] & \Sel_{\Sigma'}^{\vec s}(E/K_\infty)^{H_L} \arrow[r]                     & {H^1(G_\Sigma(K_\infty),E[p^\infty])^{H_L}} \arrow[r]&\displaystyle\bigoplus_{u\in(\Sigma-\Sigma')(K_\infty)}J_u(E/K_\infty)^{H_L}
    \end{tikzcd}
    \]
Note that the top row is exact by \cref{pro-4.6}\ref{4.6.a2}. As the middle vertical map is an isomorphism and the right vertical map is surjective (compare with the proof of \cref{prop:free-resolution}), the restriction 
\[\Sel^{\vec s}_{\Sigma'}(E/L_\infty)\to \Sel^{\vec s}_{\Sigma'}(E/K_\infty)^{H_L}\]
is injective with finite cokernel. In particular,
\[\lambda((X^{\vec s}_{\Sigma'}(E/K_\infty))^{H_L})=\lambda(X_{\Sigma'}^{\vec s}(E/L_\infty)),\quad \mu((X^{\vec s}_{\Sigma'}(E/K_\infty))^{H_L})=\mu(X_{\Sigma'}^{\vec s}(E/L_\infty)).\]
\Cref{pro-4.6}\ref{4.6.b} implies that $H^i(H_L,\Sel_{\Sigma'}^{\vec s}(E/K_\infty))$ is finite for all $i>0$. Therefore, $H^i(H_L,X_{\Sigma'}^{\vec s}(E/K_\infty))$ is finite for all $i>0$. In particular it has vanishing $\mu$-invariant. We can now apply \cref{hachimorio-sharifi-2} to conclude that
\begin{align}\label{herbrand-p-tor}h(X^{\vec s}_{\Sigma'}(E/K_\infty){[p^\infty]})=1, \quad \mu(X_{\Sigma'}^{\vec s}(E/K_\infty))=p\mu(X_{\Sigma'}^{\vec s}(E/L_\infty))\end{align}
which proves the claim on $\mu$-invariants.

It remains to show the claim on $\lambda$-invariants. Using again \cref{pro-4.6}\ref{4.6.a2} we deduce that 
\[h(\Sel_{\Sigma'}^{\vec s}(E/K_\infty))=\frac{h(H^1(G_\Sigma(K_\infty),E[p^\infty]))}{h\left(\bigoplus_{v\in (\Sigma-\Sigma')(K_\infty)}J_v(E/K_\infty)\right)}\]
By \cite[3.2 A und B]{Greenberg1999}, $h(H^1(G_\Sigma(K_\infty),E[p^\infty]])=h(J_v(E/K_\infty))=1$ for all places in $(\Sigma_1-\Sigma')(K_\infty)$ and for all ordinary primes above $p$. For the supersingular primes $h(J_v(E/K_\infty))=1$ by \cref{supersingular-herbrand}. Thus,
\begin{equation} \label{eq:herbrand-X}
    1=h(X_{\Sigma'}^{\vec s}(E/K_\infty)).
\end{equation}

Let $Z=X_{\Sigma'}^{\vec s}(E/K_\infty)/X_{\Sigma'}^{\vec s}(E/K_\infty){[p^\infty]}$. Then $Z$ is $\Z_p$-free and $\lambda(Z)=\lambda(X_{\Sigma'}^{\vec s}(E/K_\infty))$. For the $H_L$-coinvariants $Z_{H_L}$, we have $\lambda(Z_{H_L})=\lambda(X_{\Sigma'}^{\vec s}(E/L_\infty))$ and $h(Z)=1$ by \eqref{herbrand-p-tor} and \eqref{eq:herbrand-X}. We can now conclude as in \cite{hachimori-sharifi} that 
\[\lambda(Z)=p(\lambda(Z_{H_L})-v_p(h(Z)))+v_p(h(Z))=p\lambda(Z_{H_L}),\]
which concludes the proof.
\end{proof}
\begin{thm} \label{thm:Kida}
    Assume that $\Sel^{\vec s}(E/K_\infty)$ is $\Lambda$-cotorsion and that $\theta(X^{\vec s}(E/K_\infty))\le 1$. 
    Let $P_1\subset \Sigma'$ be the primes where $E$ has split multiplicative reduction and let $P_2$ be the set of primes in $\Sigma'$ where $E$ has good reduction and $E(K)[p]\neq 0$. Then we have
    \[\lambda(X^{\vec s}(E/K_\infty))=[K_\infty:L_\infty]\lambda(X^{\vec s}(E/L_\infty))+\sum_{v\in P_1}(e_v-1)+2\sum_{w\in P_2}(e_v-1)\] and
    \[\mu(X^{\vec s}(E/K_\infty))=[K_\infty:L_\infty]\mu(X^{\vec s}(E/L_\infty)).\]
\end{thm}
\begin{proof}
    By \cref{pro-4.6}\ref{4.6.a1} we have that $\Sel_{\Sigma'}^{\vec s}(E/K_\infty)$ is $\Lambda$-cotorsion and that 
    \[\lambda(X^{\vec s}(E/K_\infty))=\lambda(X_{\Sigma'}^{\vec s}(E/K_\infty))-\sum_{v\in \Sigma'(K_\infty)}\lambda(J_v(E/K_\infty)) \]
    and 
 \[\mu(X^{\vec s}(E/K_\infty))=\mu(X_{\Sigma'}^{\vec s}(E/K_\infty))-\sum_{v\in \Sigma'(K_\infty)}\mu(J_v(E/K_\infty)). \]
 By \cref{pro-4.6}\ref{4.6.b} we can apply \cref{kida-for-non-primitive} and obtain
 \begin{align*}\lambda(X^{\vec s}(E/K_\infty))&=[K_\infty:L_\infty]\lambda(X^{\vec s}(E/L_\infty))\\&+[K_\infty:L_\infty]\sum_{v\in \Sigma'(L_\infty)}\lambda(J_v(E/L_\infty))-\sum_{v\in \Sigma'(K_\infty)}\lambda(J_v(E/K_\infty)) \end{align*}
 and
  \begin{align*}\mu(X^{\vec s}(E/K_\infty))&=[K_\infty:L_\infty]\mu(X^{\vec s}(E/L_\infty))\\&+[K_\infty:L_\infty]\sum_{v\in \Sigma'(L_\infty)}\mu(J_v(E/L_\infty))-\sum_{v\in \Sigma'(K_\infty)}\mu(J_v(E/K_\infty)). \end{align*}
  The claim now follows by analysing the terms $\lambda(J_v(E/K_\infty))$ as in \cite[proof of Proposition~5.2]{mengfai} using \cite{Greenberg1989-padic} and \cite{HachimoriMatsuno} and from the fact that $\mu(J_w(E/K_\infty))=0$ \cite[Proposition 2]{Greenberg1999}.
\end{proof}

\section{Integrality of characteristic elements} \label{sec:integrality}
In this section, we generalise Lim's integrality results on characteristic elements of signed Selmer groups \cite[\S5.2]{mengfai}.

Let $\mathcal Q(\G)$ denote the total ring of quotients of the Iwasawa algebra $\Lambda(\G)$.
Let $\partial:K_1(\mathcal Q(\G))\to K_0(\Lambda(\G),\mathcal Q(\G))$ denote the connecting homomorphism in the localisation exact sequence of relative $K$-theory; for details, we refer to \cite[p.~29ff.]{Sujatha}.
For a finitely generated $\Lambda(\G)$-module $Y$ that is torsion over $\Lambda$ and has projective dimension $\mathrm{pd}_{\Lambda(\G)} Y\le1$, a characteristic element is an element $\xi_Y\in K_1(\mathcal Q(\G))$ whose image $\partial(\xi_Y)\in K_0(\Lambda(\G),\mathcal Q(\G))$ agrees with the class of $Y$ in the relative $K_0$-group.

Let $n_0$ be a large enough integer such that $\Gamma_0\colonequals\Gamma^{p^{n_0}}$ is central in $\G$, and let $\Lambda(\Gamma_0)\colonequals\Z_p\llbracket\Gamma_0\rrbracket \subset\Lambda(\G)$ denote the corresponding Iwasawa algebra.
{Recall that a $\Lambda(\Gamma_0)$-order $\mathfrak M$ in $\mathcal Q(\G)$ is called a graduated order if there exist orthogonal indecomposable idempotents $e_1,\ldots,e_t\in \mathfrak M$ such that $e_i\mathfrak M e_i$ is a maximal order in $e_i\mathcal Q(\G)e_i$ for each $i=1,\ldots t$. In particular, every maximal order is graduated \cite[Theorem~10.5.(i)]{MO}. Graduated orders over Iwasawa algebras have been studied in \cite{F-cond}.}

\begin{thm} \label{thm:integrality}
    Let $E$ be an elliptic curve satisfying (S1), (S2) and (S3), and suppose that the conditions of \cref{prop:free-resolution} hold. Let $\xi_{E,\Sigma'}$ denote a characteristic element of $X^{\vec s}_{\Sigma'}(E/K_\infty)$. Then for every {graduated} $\Lambda(\Gamma_0)$-order $\mathfrak M$ of $\mathcal Q(\G)$ containing $\Lambda(\G)$, we have
    \[\xi_{E,\Sigma'} \in \Image\left(\mathfrak M\cap \mathcal Q(\G)^\times\to K_1(\mathcal Q(\G))\right).\]
\end{thm}
\begin{proof}
    {The skew fields occurring in the Wedderburn decomposition of $\mathcal Q(\G)$ are given explicitly in \cite[Theorem~4.12]{W}. This shows that $\mathcal Q(\G)$ meets the conditions of \cite[Proposition~2.7]{F-cond}, which provides a description of graduated orders.}
    Moreover, it is shown in \cite[Proposition~6.2]{EpAC} that Nichifor--Palvannan's dimension reduction argument \cite[Proposition~2.13]{NichiforPalvannan} can be generalised to such rings.
    
    Since \cref{prop:free-resolution} shows that $X^{\vec s}_{\Sigma'}(E/K_\infty)$ admits a free resolution of length $1$, the assertion can be proven by the same argument as in {\cite[Corollary~4.3]{F-cond}, which follows along the lines of \cite[Theorem~1]{NichiforPalvannan} and \cite[Corollary~7.6]{EpAC}}.
\end{proof}

\begin{cor}
    Keep the assumptions of \cref{thm:integrality}, and further assume that $\Phi$ contains no places at which $E$ has either split multiplicative reduction or good reduction with $E(K_{\infty,w})[p]\ne 0$ (i.e. $P_1=P_2=\emptyset$ in the notation of \cref{thm:Kida}). Let $\xi_{E}$ denote a characteristic element of $X^{\vec s}(E/K_\infty)$. Then for every {graduated} $\Lambda(\Gamma_0)$-order $\mathfrak M$ containing $\Lambda(\G)$, we have
    \[\xi_{E} \in \Image\left(\mathfrak M\cap \mathcal Q(\G)^\times\to K_1(\mathcal Q(\G))\right).\] 
\end{cor}
\begin{proof}
    The proof is identical to \cite[Proposition~5.5]{mengfai}: indeed, the short exact sequence of \cref{pro-4.6}\ref{4.6.a1} combined with the corank analysis in the proof of \cite[Proposition~5.2]{mengfai} shows that
    \[\left[X^{\vec s}(E/K_\infty)\right]=\left[X^{\vec s}_{\Sigma'}(E/K_\infty)\right]\in K_0(\Lambda(\G),\mathcal Q(\G))\]
    under setting $\Sigma'\colonequals\Phi$. The claim now follows from \cref{thm:integrality}.
\end{proof}

\section{Behavior of Iwasawa invariants under congruences} \label{sec:cong}
Let $E_1$ and $E_2$ be elliptic curves defined over $F$ and assume that $E_1[p]\cong E_2[p]$ as $G_K$-modules. 
{Let $\Sigma$ be a finite set of places of $F$ satisfying \eqref{eq:Sigma-conditions}, containing the places at which $E_1$ or $E_2$ has bad reduction.}
Recall that $\Sigma_1=\Sigma-\Sigma_p$ is the set of non-$p$-adic places in $\Sigma$.
\begin{def1} \label{def:Sel-E-p}
Assume that $X^{\vec s}(E_i/K_\infty)$ is $\Lambda$-torsion for $1\le i\le 2$.
{For $E\in\{E_1,E_2\}$,} we define the $p$-primary signed Selmer group as
\[\Sel^{\vec s}(E[p]/K_\infty)=\ker\left(H^1(G_\Sigma(K_\infty),E[p])\to \bigoplus_{v\in \Sigma(K_\infty)}J_v(E[p]/K_\infty)\right),\]
where $J_v(E[p]/K_\infty)$ is defined by case distinction: 
\[J_v(E[p]/K_\infty) = \begin{cases}
    H^1(K_{\infty,v},E[p]) & \text{if } v\in \Sigma_1(K_\infty), \\
    \frac{H^1(K_{\infty,v},E[p])}{E(K_{\infty,v})/pE(K_{\infty,v})} & \text{if } v\in \Sigma^\ord(K_\infty), \\
    \frac{H^1(K_{\infty,v},E[p])} {\widehat{E}^\pm(K_{\infty,v})/p\widehat{E}^\pm(K_{\infty,v})} & \text{if } v\in \Sigma^\ss(K_\infty).
\end{cases}\]
We define the non-primitive version as
\[\Sel_{\Sigma_1}^{\vec s}(E[p]/K_\infty)=\ker\left(H^1(G_\Sigma(K_\infty),E[p])\to \bigoplus_{v\in (\Sigma-\Sigma_1)(K_\infty)}J_v(E[p]/K_\infty)\right),\]
\end{def1}
\begin{lemma}
\label{lemma-p-aussen-innen}
    Then there is a natural isomorphism
    \[\Sel_{\Sigma_1}^{\vec s}(E_i[p]/K_\infty)\cong \Sel^{\vec s}_{\Sigma_1}(E_i/K_\infty)[p].\]
\end{lemma}
\begin{proof}
    We have a commutative diagram
      \[\begin{tikzcd}
    0 \arrow[r] & \Sel_{\Sigma_1}^{\vec s}(E_i[p]/K_\infty) \arrow[r] \arrow[d] & {H^1(G_{\Sigma}(K_\infty),E_i[p])} \arrow[r] \arrow[d, Rightarrow, no head] & {\displaystyle\bigoplus_{v\in (\Sigma-\Sigma_1)(K_\infty)} J_v(E_i[p] / {K_{\infty}})} \arrow[d,"\oplus_vd_v"]  \\
    0 \arrow[r] & \Sel^{\vec s}_{\Sigma_1}(E_i/K_\infty)[p] \arrow[r] & {H^1(G_{\Sigma}(K_\infty),E_i[p^\infty])[p]} \arrow[r]  & {\displaystyle\bigoplus_{v\in (\Sigma-{\Sigma_1})(K_\infty)} J_v(E_i[p^\infty] / {K_{\infty}})}
    \end{tikzcd}\]
    Here $J_v(E_i[p^\infty] / K_{\infty,w})$ is defined analogously to \cref{def:Sel-E-p}.
    One can easily check that $d_v$ is injective for ordinary primes $v$ (see also \cite[proof of Proposition 2.8]{GreenbergVatsal}). For supersingular $v$, \cref{ss-p-primary} implies that $d_v$ is injective. Thus, the left vertical map is surjective by the snake Lemma.
\end{proof}
\begin{prop}
\label{formula-for-invariants-non-primitive}Assume that $E_1[p]\cong E_2[p]$ as $G_K$-modules.
    Assume that $X_{\Sigma_1}^{\vec s}(E_1/K_\infty)$ is $\Lambda$-torsion and finitely generated over $\Z_p$. Then the same is true for $X_{\Sigma_1}^{\vec s}(E_2/K_\infty)$ and the $\lambda$-invariants are the same. 
\end{prop}
\begin{proof}
As $E_1[p]\cong E_2[p]$, we obtain an isomorphism 
\[X_{\Sigma_1}^{\vec s}(E_1[p]/K_\infty)\cong X_{\Sigma_1}^{\vec s}(E_2[p]/K_\infty).\]
By \cref{lemma-p-aussen-innen} we have isomorphisms
\[X_{\Sigma_1}^{\vec s}(E_i[p]/K_\infty)\cong X_{\Sigma_1}^{\vec s}(E_i/K_\infty)\big/pX^{\vec s}_{\Sigma_1}(E_i/K_\infty)\]
for $1\le i\le 2$. Thus, if $X_{\Sigma_1}^{\vec s}(E_1/K_\infty)$ is finitely generated over $\Z_p$, the same is true for $X_{\Sigma_1}^{\vec s}(E_2/K_\infty)$. 

It remains to prove the claim concerning $\lambda$-invariants. By \cref{pro-4.6}\ref{4.6.d} $X_{\Sigma_1}^{\vec s}(E_i/K_\infty)$ does not contain a non-trivial finite submodule. Thus, if it is finitely generated over $\Z_p$, we have
\[\lambda(X_{\Sigma_1}^{\vec s}(E/K_\infty))=v_p\left(\left|X_{\Sigma_1}^{\vec s}(E/K_\infty)\big/pX_{\Sigma_1}^{\vec s}(E/K_\infty)\right|\right).\]
The claim now follows from the two isomorphisms above.
\end{proof}
\begin{thm} \label{thm:congruent}
    Assume that $E_1[p]\cong E_2[p]$ as $G_K$-modules.
    Assume that $X^{\vec s}(E_1/K_\infty)$ is $\Lambda$-torsion and finitely generated over $\Z_p$. Then the same is true for $X^{\vec s}(E_2/K_\infty)$ and we get the following equality of $\lambda$-invariants
    \[\lambda(X^{\vec s}(E_1/K_\infty))+\sum_{v\in \Sigma_1}\lambda(J_v(E_1/K_\infty)^\vee )=\lambda(X^{\vec s}(E_2/K_\infty))+\sum_{v\in \Sigma_1}\lambda(J_v(E_2/K_\infty)^\vee).\]
\end{thm}
{Similar results have been achieved by a number of authors. 
Greenberg--Vatsal compared algebraic and analytic Iwasawa invariants of modular elliptic curves under the assumption that $E_1[p]\cong E_2[p]$ is irreducible \cite[Theorem~1.4]{GreenbergVatsal}. 
B.D.~Kim established equality of $\lambda$-invariants of non-primitive plus/minus Selmer groups of congruent elliptic curves over $\Q$ \cite[Corollary~2.13]{BDK-invariants}.
Ahmed--Aribam--Shekhar studied the parity of $\lambda$-invariants and root numbers of congruent elliptic curves defined over $\Q$ while assuming irreducibility \cite{AAS}. Our \cref{thm:congruent} can be seen as a generalisation of these results on algebraic Iwasawa invariants to our setting.}

{The question has also been studied for anticyclotomic Selmer groups by Hatley--Lei \cite[Proposition 5.4]{HL-I} \cite[Theorem 4.6]{HL-II}.
In the context of ordinary modular forms, a similar result is due to Emerton--Pollack--Weston \cite[Theorem 2]{EmertonPollackWeston}. In the supersingular case, Hatley--Lei compared $\lambda$-invariants of signed Selmer groups of two modular forms of the same even weight that are congruent modulo $p$ \cite[Theorem 4.6]{HatleyLei}.
}

\begin{proof}[{Proof of \cref{thm:congruent}}]
By \cite[Proposition 2]{Greenberg1989-padic} the modules $J_v(E_i/K_\infty)^\vee$ are finitely generated over $\Z_p$. Therefore \cref{pro-4.6}\ref{4.6.a1} implies that $X_{\Sigma_1}^{\vec s}(E_i/K_\infty)$ is finitely generated over $\Z_p$ if and only if the same is true for $X^{\vec s}(E_i/K_\infty)$. In this case we obtain the following equality of Iwasawa invariants:
\[\lambda(X^{\vec s}(E_i/K_\infty))+\sum_{v\in \Sigma_1}\lambda(J_v(E_i/K_\infty)^\vee )=\lambda(X_{\Sigma_1}^{\vec s}(E_i/K_\infty)).\]
The desired claim now follows from \cref{formula-for-invariants-non-primitive}.
\end{proof}

\printbibliography

\end{document}